\documentclass{amsart}

\usepackage{amsmath,amssymb,amsthm}
\usepackage{booktabs}
\usepackage{float}
\usepackage[colorlinks=true,linkcolor=blue,citecolor=blue,urlcolor=blue]{hyperref}
\usepackage[capitalise,noabbrev,nameinlink]{cleveref}

\providecommand{\F}{\mathbb F}
\providecommand{\Z}{\mathbb Z}

\providecommand{\PP}{\mathbb P}

\providecommand{\Aut}{\operatorname{Aut}}
\providecommand{\PAut}{\operatorname{PAut}}

\providecommand{\Mon}{\operatorname{Mon}}

\providecommand{\Supp}{\operatorname{Supp}}
\providecommand{\Hull}{\operatorname{Hull}}
\providecommand{\rank}{\operatorname{rank}}
\providecommand{\diag}{\operatorname{diag}}
\providecommand{\wt}{\operatorname{wt}}

\providecommand{\Ram}{\operatorname{Ram}}

\providecommand{\XX}{\mathcal X}
\providecommand{\YY}{\mathcal Y}
\providecommand{\CC}{\mathcal C}
\providecommand{\HH}{\mathcal H}

\newcommand{\LCD}{\mathrm{LCD}}
\newcommand{\SSA}{\mathrm{SSA}}
\newcommand{\PCE}{\mathrm{PCE}}
\newcommand{\LEP}{\mathrm{LEP}}

\theoremstyle{plain}
\newtheorem{thm}{Theorem}[section]
\newtheorem{prop}[thm]{Proposition}
\newtheorem{lem}[thm]{Lemma}
\newtheorem{cor}[thm]{Corollary}

\theoremstyle{definition}
\newtheorem{defn}[thm]{Definition}

\newtheorem{rem}[thm]{Remark}

\crefname{thm}{Thm.}{Thms.}
\Crefname{thm}{Thm.}{Thms.}
\crefname{prop}{Prop.}{Props.}
\Crefname{prop}{Prop.}{Props.}
\crefname{lem}{Lem.}{Lems.}
\Crefname{lem}{Lem.}{Lems.}
\crefname{cor}{Cor.}{Cors.}
\Crefname{cor}{Cor.}{Cors.}
\crefname{equation}{Eq.}{Eqs.}
\Crefname{equation}{Eq.}{Eqs.}
\crefname{defn}{Def.}{Definitions}
\Crefname{defn}{Def.}{Definitions}
\crefname{exa}{Exa.}{Examples}
\Crefname{exa}{Exa.}{Examples}
\crefname{rem}{Rem.}{Remarks}
\Crefname{rem}{Rem.}{Remarks}
\crefname{prob}{Prob.}{Problems}
\Crefname{prob}{Prob.}{Problems}

\providecommand{\Sym}{\operatorname{Sym}}
\providecommand{\Can}{\operatorname{Can}}
\providecommand{\Stab}{\operatorname{Stab}}
\providecommand{\Nmon}{N_{\mathrm{mon}}}
\providecommand{\Maut}{\operatorname{MAut}}
\providecommand{\Per}{\operatorname{Per}}
\providecommand{\sh}{\operatorname{Sh}}
\providecommand{\nul}{\operatorname{null}}
\providecommand{\im}{\operatorname{im}}
\providecommand{\one}{\mathbf 1}
\providecommand{\PGL}{\operatorname{PGL}}

\usepackage[margin=1.15in]{geometry}

\begin{document}

\title{Code Equivalence and Automorphism Groups of Weighted Superelliptic Codes}

\author{L. Beshaj}
\address{Department of Mathematical Sciences, United States Military Academy, West Point, NY 10996}

\author{T. Shaska}
\address{Department of Computer Science \& Engineering, Oakland University, Rochester, MI 48309}



\begin{abstract}
We study permutation and monomial equivalence for one-point codes on weighted superelliptic curves.  On strata where $H_m\simeq \Z/n\Z\times\Z/m\Z$ acts regularly on the evaluation set, we prove that these codes are Cartesian character codes and determine their primal and dual minimum distances.  We compute the projective monomial normalizer of the regular action and classify the compatible permutation and monomial equivalences through the induced actions on the defining sets.  When $d > n$ and $m \geq 2$, the common permutation automorphism group of the full filtration is exactly $H_m$, and the regular action can be recovered from its two lowest nonconstant members.  We determine the full permutation automorphism group in closed form at the top of the filtration and give a field-independent finite certificate for $\PAut ( \CC_s ) = H_m$ in the remaining primitive cases treated here.  The hull dimension is bounded by the genus on every regular stratum, apart from the type $(3,3,3)$.  Explicit examples show that length, dimension, and hull dimension do not determine monomial equivalence, and that the hull profile together with the regular automorphism subgroup does not determine the curve up to isomorphism.  As an application, we determine the corresponding hull-based invariants and punctured signatures used in support-splitting methods.
\end{abstract}

\maketitle

\section{Introduction}\label{sec-intro}

A linear code is a subspace $C \subseteq \F_q^N$, whose elements are written as row vectors; for an $N \times N$ matrix $M$ put $CM = \{ cM : c \in C \}$.  A monomial matrix is the product of a permutation matrix and an invertible diagonal matrix.  Two linear codes $C, C' \subseteq \F_q^N$ are permutation equivalent if $C' = CP$ for a permutation matrix $P$, and linearly equivalent if $C' = CM$ for a monomial matrix $M$.  Deciding the first is the permutation code equivalence problem ($\PCE$); deciding the second is the linear equivalence problem ($\LEP$).  Both are classical, see \cite{mac-sloane}.  The first is at least as hard as graph isomorphism and is not NP-complete unless the polynomial hierarchy collapses, see \cite{petrank-roth}.  The second underlies the signature scheme LESS and its variants, see \cite{less,persichetti-santini}.

Sendrier's support splitting algorithm solves $\PCE$ by comparing, coordinate by coordinate, the weight enumerator of the hull $\Hull (C) = C \cap C^\perp$.  Its cost is polynomial in $N$ and exponential in $h = \dim \Hull (C)$, see \cite{se-ssa}.  When $h = 0$, $\PCE$ reduces in polynomial time to weighted graph isomorphism, see \cite{ba-ot-sa}.  A random code has hull of bounded dimension, see \cite{se-hull}, so for random codes the exponent is bounded.  For a structured family of codes, $h$ must be computed for each member.

This paper computes $h$ for the one-point codes $\CC_s = C_L ( D, s D_\infty )$ on weighted superelliptic curves $y^n = f(x)$.  These are the weighted algebraic geometry codes of \cite{qwag}, cut out by weighted forms of degree $s$ on the weighted plane model of the curve; their hulls are computed in \cite{mezinaj-shaska-hulls}.  The exact formula of \cite{mezinaj-shaska-hulls} writes $\dim \Hull ( \CC_s )$ as a lattice count $\Phi ( \nu (s) )$ plus an excess $\varepsilon ( G_s, A_s )$, and shows that the excess vanishes on an explicit interval of degrees, the Riemann--Roch window.  On the window $h$ is an affine function of $s$ with slope $\pm c$, $c = \gcd (n, \deg f)$, depending on the curve only through its number of affine rational points.  Beyond an explicit threshold $s^+$ the dual is zero and $h = 0$.  On the strata where a subgroup $H_m \simeq \Z / n \Z \times \Z / m \Z$ of the automorphism group acts regularly on the evaluation set, the window is empty, the codes are ideals in $\F_q [ H_m ]$, and $h$ is the number of characters in the defining set $T_s$ whose inverse is not in $T_s$.  Each of these formulas can be evaluated.

The same automorphisms are permutation automorphisms of the codes.  The support splitting algorithm separates coordinates only up to the orbits of the permutation automorphism group, see \cite{se-ssa}, so $H_m$ is a second parameter of the algorithm; on a regular stratum its signatures are constant.  There the monomial matrices compatible with the regular action form the group $\F_q^* ( H_m \rtimes \Aut ( H_m ) ) \widehat H_m$, and the compatible permutation and monomial equivalences are classified by the action of $\Aut ( H_m )$ and of $\widehat H_m \rtimes \Aut ( H_m )$ on the defining set $T_s$.  The compatible class of $\CC_s$ depends only on $(n, d, m, s)$.  The regular action also makes the codes Cartesian: both minimum distances are explicit, a single puncture lowers a nonzero hull by exactly one dimension, so the initial signature costs $q^{h-1}$ words rather than $q^h$, and the two codes $\CC_{n/c}$, $\CC_{d/c}$ determine the regular action while the full filtration has permutation automorphism group exactly $H_m$.  On every regular stratum the exponent is bounded by $g_\XX$, with a single exception of type $(3, 3, 3)$, whereas on the totally split locus it grows linearly in $N$.  Two codes over $\F_{13}$ with equal length, dimension and hull dimension are separated by their distances, hence are not monomially equivalent; for the codes in the tables the full permutation automorphism groups are computed over the whole degree range; they are $H_m$ extended by the stabilizer of the defining set except at three explicit types of degrees, where the group is larger and is determined, and at every tabulated degree with trivial stabilizer the equality $\PAut ( \CC_s ) = H_m$ is certified by a finite computation on the character group that does not depend on the field.  Two non-isomorphic curves over $\F_{19}$ with the same hull profile $s \mapsto \dim \Hull ( \CC_s )$ and the same automorphism subgroup show that these invariants do not determine the curve up to isomorphism.

The paper \cite{mezinaj-shaska-hulls} proves the hull formulas and the character description of these codes; everything in it that is used here is recalled in \cref{sec-preliminaries} without proof.  The present paper is the equivalence and automorphism theory built on that input: the Cartesian description and the exact distances (\cref{sec-cartesian}), the constancy of the signatures and the exact effect of punctures (\cref{sec-signatures}), the monomial normalizer and the classification of compatible permutation and monomial equivalence (\cref{sec-compatible-equivalence}), the automorphism groups of the filtration and the recovery of the regular action from two of its members (\cref{sec-rigidity}), and the determination of the full permutation automorphism groups together with the example of two non-isomorphic curves with the same invariants (\cref{sec-automorphisms}).

The conclusions concern $\PCE$.  Hull dimension is not invariant under monomial equivalence: for $q > 3$ every linear code is monomially equivalent to a linear complementary dual ($\LCD$) code, one with $\Hull (C) = 0$, see \cite{ca-me-ta-qi-pe}.  The permutation and diagonal parts of a monomial equivalence must therefore be treated separately, and nothing here asserts the security of a signature scheme.

\section{Preliminaries}\label{sec-preliminaries}

This section fixes the code-equivalence notation and recalls the results of \cite{mezinaj-shaska-hulls} used below.  Results taken directly from that paper are stated with precise citations and are not reproved; their immediate consequences are retained with proofs.

\subsection{Code equivalence and support splitting}\label{sec-equivalence}

Let $C \subseteq \F_q^N$ be a linear code, equipped with the Euclidean inner product $\langle u, v \rangle = u v^t$.  Since every permutation matrix is monomial, permutation equivalence is a special case of linear equivalence.  Permutation matrices preserve the Euclidean inner product, whereas diagonal scalings need not.

\begin{defn}\label{def-ssa-exponent}
For a linear code $C \subseteq \F_q^N$ put
\[
e_{\SSA} (C) = \dim \Hull (C) .
\]
We call $e_{\SSA} (C)$ the support-splitting exponent of $C$.
\end{defn}

The number $q^{e_{\SSA} (C)}$ is the exponential factor in the cost of the hull weight-enumerator signature.  It is not a running time: the polynomial factor and the cost of any improved enumeration method are not included, and for a code with a transitive permutation automorphism group the initial signature itself is the enumerator of a code of dimension $e_{\SSA} (C) - 1$.

Three facts about this exponent are used throughout.  If $C' = CP$ with $P$ a permutation matrix, then $( CP )^\perp = C^\perp P$ and $\Hull ( CP ) = \Hull (C) P$, so $e_{\SSA}$ is a permutation invariant.  The hull signature of the support splitting algorithm is computed in $q^{e_{\SSA} (C)} \operatorname{poly} (N)$ field operations, see \cite{se-ssa}.  If $e_{\SSA} (C) = 0$, the instance of $\PCE$ reduces in polynomial time to weighted graph isomorphism, see \cite{ba-ot-sa}.

\subsection{Weighted superelliptic evaluation codes}\label{sec-family}

We use the notation of \cite{mezinaj-shaska-hulls}.  The curve $\XX$ is taken with its weighted plane model in $\PP^2_{(1, n/c, d/c)}$, $c = \gcd (n, d)$, and $\CC_s$ is the code of weighted forms of degree $s$ evaluated on $\YY$, as in \cite{qwag}; by \cite[Lemma~2.5]{mezinaj-shaska-hulls} this is $C_L ( D, s D_\infty )$.  Unless another hypothesis is explicitly stated, we impose the following standing hypotheses throughout.

\begin{enumerate}
\item[(H1)] $\XX$ is the smooth projective model of
\[
\XX : \quad y^n = f(x)
\]
over $\F_q$, where $n \geq 2$, $p \nmid n$, and $f \in \F_q [x]$ is separable of degree $d \geq 3$ and has no root in $\F_q$.
\item[(H2)] $n \mid q-1$.
\item[(H3)] $\YY$ is the set of all affine $\F_q$-rational points of $\XX$, and $\YY \neq \varnothing$.
\end{enumerate}

Put $c = \gcd (n, d)$.  The reduced divisor $D_\infty$ above $x = \infty$ has degree $c$, and
\[
g_\XX = \frac{(n-1)(d-1) + 1 - c}{2} \; \geq \; 1 .
\]
Let $D = \sum_{P \in \YY} P$ and $N = | \YY |$.  Since $f$ has no $\F_q$-rational root, every affine fibre contains either $n$ rational points or none.  Hence, if
\[
\Delta = \{ a \in \F_q : f(a) \notin ( \F_q^* )^n \} , \qquad \delta = | \Delta | ,
\]
then $N = n ( q - \delta )$.  For $s \geq 0$ put
\[
G_s = s D_\infty , \qquad \CC_s = C_L ( D, G_s ) , \qquad M = \frac{(q-1) n}{c} - 1 , \qquad \mu (s) = \min \{ s, M - s \} .
\]
Thus $\CC_0$ is the line spanned by the all-one vector.  The differential $- dx / ( x^q - x )$ has residue one at the points of $D$, and therefore the Euclidean dual of $\CC_s$ is $C_L ( D, A_s )$ for
\[
A_s = (n-1) \Ram - E + ( M - s ) D_\infty .
\]
Here $\Ram$ is the reduced ramification divisor over the roots of $f$, and $E$ is the effective divisor formed by the non-rational affine fibres, so that $\deg E = n \delta$.  These hypotheses exclude curves with $\F_q$-rational branch points; the completely split branch locus requires the modified dual divisor of \cite{mezinaj-shaska-hulls} and is not treated here.

\begin{defn}\label{def-window}
For $s \geq 1$ set
\[
\nu (s) = \mu (s) - \frac{n \delta}{c} , \qquad t (s) = 2 g_\XX - 2 - c \, \mu (s) + n \delta ,
\]
and for $\nu \in \Z$ set
\[
\Phi ( \nu ) = \sum_{j=0}^{n-1} \max \left\{ 0, \left\lfloor \frac{c \nu - d j}{n} \right\rfloor + 1 \right\} .
\]
For $\nu < 0$ this expression is zero.  The Riemann--Roch window is
\[
W ( \XX ) = \{ s \geq 1 : t (s) < 0 \} .
\]
\end{defn}

The integer $t (s)$ is the degree of $K - G_s \wedge A_s$, where $G_s \wedge A_s$ is the meet of the two divisors.  The hull formula of \cite[Theorem~4.13 and Corollaries~3.5 and~4.14]{mezinaj-shaska-hulls} reads, in the present notation: under \textup{(H1)}--\textup{(H3)}, for $s \geq 1$ with $c \, \nu (s) < N$,
\begin{equation}\label{eq-hull-input}
\dim \Hull ( \CC_s ) = \Phi ( \nu (s) ) + \varepsilon ( G_s, A_s ) , \qquad \varepsilon ( G_s, A_s ) \geq 0 ,
\end{equation}
the excess vanishes for every $s \in W ( \XX )$, and if $0 \leq t (s) \leq 2 g_\XX - 2$ then $\dim \Hull ( \CC_s ) \leq g_\XX + 1$.  On the window this gives
\begin{equation}\label{eq-window}
e_{\SSA} ( \CC_s ) = c \, \mu (s) - n \delta + 1 - g_\XX \; \geq \; g_\XX , \qquad s \in W ( \XX ) ,
\end{equation}
the inequality because $c \, \nu (s) > 2 g_\XX - 2$ there; in particular no code $\CC_s$ with $s \in W ( \XX )$ is $\LCD$.  At the other end, with
\begin{equation}\label{eq-splus}
s^+ = M + \left\lfloor \frac{(n-1) d - n \delta}{c} \right\rfloor + 1 ,
\end{equation}
one has $\CC_s^\perp = 0$ for every $s \geq s^+$, see \cite[Theorem~4.15]{mezinaj-shaska-hulls}, hence $e_{\SSA} ( \CC_s ) = 0$ and $\PCE$ for $\CC_s$ reduces in polynomial time to weighted graph isomorphism.

\begin{prop}\label{cor-asymptotics}
Assume \textup{(H1)}--\textup{(H3)} and $\delta = 0$.  Then $N = qn$, and $W ( \XX ) \neq \varnothing$ if and only if $c \lfloor M/2 \rfloor > 2 g_\XX - 2$.  In that case
\[
\max_{s \in W ( \XX )} e_{\SSA} ( \CC_s ) = \frac{N - n - c - c \theta}{2} + 1 - g_\XX , \qquad \theta \in \{ 0, 1 \} , \quad \theta \equiv M \pmod 2 .
\]
Moreover $\delta = 0$ forces $(n-1) q \leq 1 + 2 g_\XX \sqrt q$.
\end{prop}

\begin{proof}
The criterion for $W ( \XX ) \neq \varnothing$ is \cite[Proposition~4.7]{mezinaj-shaska-hulls}.  By \cref{eq-window} the maximum on the window is attained at $\mu (s) = \lfloor M/2 \rfloor$ and equals $c \lfloor M/2 \rfloor + 1 - g_\XX$.  With $\delta = 0$ one has $cM = (q-1) n - c = N - n - c$ and $c \lfloor M/2 \rfloor = ( cM - c \theta ) / 2$.  The Hasse--Weil bound gives $qn = N \leq \# \XX ( \F_q ) \leq q + 1 + 2 g_\XX \sqrt q$.
\end{proof}

For fixed $(n, d)$ the inequality $(n-1) q \leq 1 + 2 g_\XX \sqrt q$ bounds $q$.  The totally split locus is therefore finite, and the linear growth of \cref{cor-asymptotics} is a statement about $N$ on that finite locus, not an asymptotic in $q$; \cref{prop-pointcount} gives the general form of this constraint.

\subsection{Regular actions and character decompositions}\label{sec-group}

The curve automorphisms give coordinate permutations of the codes.  We use the explicit subgroup of \cite{mezinaj-shaska-hulls}, and make no assertion about the full permutation automorphism group.

\begin{defn}\label{def-mf}
For $f = \sum_i a_i x^i$ put
\[
m (f) = \gcd \bigl( q-1, \{ i : a_i \neq 0 \} \bigr) .
\]
If $m \mid m (f)$ and $\lambda \in \F_q^*$ has order $m$, set
\[
\sigma_\lambda (x, y) = ( \lambda x, y ) , \qquad \tau (x, y) = ( x, \zeta_n y ) ,
\]
and
\[
H_m = \langle \tau, \sigma_\lambda \rangle \simeq \Z / n \Z \times \Z / m \Z .
\]
\end{defn}

Under \textup{(H1)}--\textup{(H3)} and $m \mid m (f)$, the group $H_m$ preserves $D_\infty$ and $\YY$ and acts on each $\CC_s$ by coordinate permutations, so $H_m \leq \PAut ( \CC_s )$ for every $s \geq 1$, see \cite[Propositions~5.2 and~5.3]{mezinaj-shaska-hulls}.

\begin{defn}\label{def-regular-stratum}
We say that $( \XX, m )$ lies on a regular stratum if $H_m$ acts freely and transitively on $\YY$.  Then $N = | H_m | = nm$.
\end{defn}

The regular-stratum criterion of \cite[Propositions~5.3 and~5.5]{mezinaj-shaska-hulls} is as follows.  Assume \textup{(H1)}--\textup{(H3)}, let $m \mid m (f)$, let $\lambda \in \F_q^*$ have order $m$, write $f (x) = g ( x^m )$, and assume that $g (0) \notin ( \F_q^* )^n$.  Then $H_m$ acts freely on $\YY$, and it acts transitively if and only if $N = nm$, if and only if exactly one of the $(q-1)/m$ values $g ( u^m )$, with $u$ ranging over representatives of $\F_q^* / \langle \lambda \rangle$, lies in $( \F_q^* )^n$.  On such a stratum, $\delta = q - m$ and $W ( \XX ) = \varnothing$.  We refer to these as the regular-stratum hypotheses.

On a regular stratum, after choosing one evaluation point, the ambient space $\F_q^N$ is the regular representation of $H_m$.  Each monomial $x^i y^j$ has character $( i \bmod m, j \bmod n )$, and the code is determined by a set of characters.

\begin{defn}\label{def-Ts}
On a regular stratum put
\[
T_s = \{ ( i \bmod m, j \bmod n ) : x^i y^j \in B_s \} \subseteq \Z_m \times \Z_n ,
\]
where
\[
B_s = \{ x^i y^j : i \geq 0 , \ 0 \leq j \leq n-1 , \ n i + d j \leq c s \}
\]
is the weighted monomial basis of $L ( s D_\infty )$, and $- T_s = \{ ( -a, -b ) : (a, b) \in T_s \}$.
\end{defn}

The character formula of \cite[Lemma~5.7 and Theorem~5.8]{mezinaj-shaska-hulls} states that on a regular stratum $\CC_s$ is an ideal of $\F_q [ H_m ]$ with defining set $T_s$, that $\dim \CC_s = | T_s |$, and that
\begin{equation}\label{eq-character}
\Hull ( \CC_s ) = \bigoplus_{\chi \in T_s \setminus ( - T_s )} ( \F_q^N )_\chi , \qquad e_{\SSA} ( \CC_s ) = | T_s | - | T_s \cap ( - T_s ) | ,
\end{equation}
so that $\CC_s$ is $\LCD$ if and only if $T_s = - T_s$.  This is the analogue for $\Z_m \times \Z_n$ of the criterion for cyclic codes in terms of a defining set and its negative; compare \cite{sk}.  Since at most $t = | H_m [2] | = \gcd (2, m) \gcd (2, n)$ characters are their own negatives, \cref{eq-character} gives the uniform bound $e_{\SSA} ( \CC_s ) \leq ( nm - t ) / 2$ for every $s \geq 1$ \cite[Proposition~5.10]{mezinaj-shaska-hulls}.

\section{Cartesian regular strata: distances and genus bounds}\label{sec-cartesian}

The regular action identifies the evaluation set with a Cartesian grid, and the defining set with a downward closed set of lattice points.

\begin{prop}
\label{prop-grid}
Assume the regular-stratum hypotheses of \cref{sec-group} and that $H_m$ acts transitively on $\YY$.  Let $a \in \F_q^*$ be such that $\{ x \in \F_q : f (x) \in ( \F_q^* )^n \} = a \mu_m$, and let $b \in \F_q^*$ satisfy $b^n = f (a)$.  Then
\[
\YY = a \mu_m \times b \mu_n .
\]
Index the coordinates by $( \lambda, \zeta ) \mapsto ( a \lambda, b \zeta )$ and let $\chi_{i, j} ( \lambda, \zeta ) = \lambda^i \zeta^j$.  Then $\CC_s = \bigoplus_{(i, j) \in T_s} \F_q \chi_{i, j}$ with
\[
T_s = \{ (i, j) \in [ 0, m - 1 ] \times [ 0, n - 1 ] : n i + d j \leq c s \} ,
\]
a downward closed subset of the rectangle, and
\[
k_s = | T_s | = \sum_{j = 0}^{n - 1} \min \Bigl\{ m , \ \max \Bigl\{ 0 , \ 1 + \Bigl\lfloor \frac{c s - d j}{n} \Bigr\rfloor \Bigr\} \Bigr\} .
\]
\end{prop}

\begin{proof}
By $m \mid m (f)$ the polynomial $f$ is constant on $a \mu_m$, so the fibre over every point of $a \mu_m$ is $b \mu_n$.  Evaluation of $x^i y^j$ gives $a^i b^j \chi_{i \bmod m, j}$, and replacing $i$ by its least nonnegative residue does not increase $n i + d j$.  Hence $T_s$ is the set of \cref{def-Ts} written with least representatives; it is downward closed because the weight $n i + d j$ is monotone in $i$ and in $j$.  The dimension is the number of lattice points, counted by rows.
\end{proof}

A subset $T \subseteq [ 0, m - 1 ] \times [ 0, n - 1 ]$ is downward closed if $(i, j) \in T$ implies $(i', j') \in T$ whenever $0 \leq i' \leq i$ and $0 \leq j' \leq j$.  The next result is a two-variable instance of the footprint-bound method for Cartesian codes of \cite{cartesian,monomial-cartesian}, whose distance formula it specializes to the weighted staircases $T_s$; the proof is included because the dual case and the equality statement are used later.

\begin{thm}
\label{thm-distance}
Let $T$ be a nonempty downward closed subset of $[ 0, m - 1 ] \times [ 0, n - 1 ]$.  On a grid of $m$ distinct abscissae and $n$ distinct ordinates, the monomial evaluation code $C_T$ has minimum distance
\[
d ( C_T ) = \min_{(i, j) \in T} ( m - i ) ( n - j ) .
\]
On the grid $a \mu_m \times b \mu_n$, if $T$ is also proper,
\[
d ( C_T^\perp ) = \min_{(i, j) \notin T} ( i + 1 ) ( j + 1 ) .
\]
\end{thm}

\begin{proof}
Let $F (x, y) \neq 0$ have support in $T$.  Write $F = \sum_{v = 0}^{j} F_v (x) y^v$ with $F_j \neq 0$, and put $i = \deg F_j$.  At least $m - i$ abscissae satisfy $F_j (x) \neq 0$.  At each such abscissa $F (x, y)$ has degree $j$ in $y$ and is nonzero at at least $n - j$ ordinates.  Its evaluated weight is at least $( m - i ) ( n - j )$, and $(i, j) \in T$.  For each $(i, j) \in T$, choose $i$ distinct abscissae $a_u$ and $j$ distinct ordinates $b_v$; the polynomial $\prod_{u = 1}^{i} ( x - a_u ) \prod_{v = 1}^{j} ( y - b_v )$ has support in $T$ and evaluated weight exactly $( m - i ) ( n - j )$.  This proves the first formula.

Put $U = \{ ( m - 1 - i, n - 1 - j ) : (i, j) \notin T \}$, a nonempty downward closed set.  In the character group, $\widehat H \setminus ( - T ) = U + ( 1, 1 )$.  Multiplication by the everywhere nonzero character $\chi_{1, 1}$ identifies $C_U$ with $C_T^\perp$ and preserves weights.  Applying the first formula to $U$ gives the second.  Scaling the two grid coordinates by $a$ and $b$ multiplies each monomial by a nonzero scalar, so both formulas hold on the coset grid.
\end{proof}

\begin{cor}\label{cor-fullspace}
On a regular stratum, $\CC_s = \F_q^N$ if and only if $c s \geq n ( m - 1 ) + d ( n - 1 )$.  For a nonzero proper member, $d ( \CC_s )$ and $d ( \CC_s^\perp )$ are given by \cref{thm-distance} with $T = T_s$.
\end{cor}

\begin{proof}
The largest weight in the rectangle is attained at $( m - 1, n - 1 )$, and by \cref{prop-grid} the code is the full space exactly when the whole rectangle lies in $T_s$.
\end{proof}

\begin{cor}
\label{cor-character-signature}
Let $\rho$ and $\rho'$ be the chosen regular actions of $H_m$ on two regular-stratum codes $\CC_s$ and $\CC'_s$, and suppose a permutation $P$ with $\CC'_s = \CC_s P$ satisfies
\[
P \rho (h) P^{-1} = \rho' ( \alpha (h) ) \qquad ( h \in H_m )
\]
for some $\alpha \in \Aut ( H_m )$.  Then the automorphism of the character group induced by $\alpha$ carries
\[
\Sigma_s = T_s \setminus ( - T_s )
\]
onto the corresponding set for $\CC'_s$.
\end{cor}

\begin{proof}
The conjugacy relation carries each isotypic line for $\rho$ to the isotypic line for $\rho'$ whose character is the transport through $\alpha$, and $P$ carries $\Hull ( \CC_s )$ onto $\Hull ( \CC'_s )$.  Apply \cref{eq-character}.
\end{proof}

\begin{rem}
Without the hypothesis that $P$ respects the chosen regular actions, the number $| \Sigma_s | = e_{\SSA} ( \CC_s )$ remains a permutation invariant (\cref{sec-equivalence}), but the labelled set $\Sigma_s$ need not be.
\end{rem}

The two bounds combine with the case $m = d$ of \cite{mezinaj-shaska-hulls} into a genus bound.

\begin{thm}
\label{thm-regular-genus}
On a regular stratum, $e_{\SSA} ( \CC_s ) \leq g_\XX$ for every $s \geq 1$, unless $(n, d, m) = (3, 3, 3)$, where $\max_s e_{\SSA} ( \CC_s ) = 2 = g_\XX + 1$.  If $m = d$ and $(n, d) \neq (3, 3)$, the bound is attained.
\end{thm}

\begin{proof}
For $m = d$ this is \cite[Theorem~5.23]{mezinaj-shaska-hulls}.  Let $m < d$.  Since $m \mid d$, one has $m \leq d/2$, and by the bound $e_{\SSA} ( \CC_s ) \leq ( nm - t ) / 2$ of \cite[Proposition~5.10]{mezinaj-shaska-hulls} it suffices to show $nm - t \leq 2 g_\XX = nd - n - d + 2 - c$.  Using $nm \leq nd/2$,
\[
2 g_\XX - nm + t \geq \frac{(n-2)(d-2)}{2} - ( c - t ) .
\]
If $c \leq 2$, then $c \leq t$, because $c = 2$ forces $2 \mid n$ and $t \geq \gcd (2, n) = 2$; the right side is nonnegative.  If $c = 3$, then $n$ and $d$ are multiples of $3$; for $(n, d) = (3, 3)$ the only proper divisor is $m = 1$, and $nm - t = 3 - 1 = 2 = 2 g_\XX$, while otherwise $(n - 2)(d - 2) \geq 4 \geq 2 ( 3 - t )$.  If $c = 4$, then $n$ is even, $t \geq 2$, and $(n-2)(d-2) \geq 4 \geq 2 ( 4 - t )$.  If $c \geq 5$, then $(n-2)(d-2) \geq (c-2)^2 \geq 2 ( c - 1 ) \geq 2 ( c - t )$.
\end{proof}

By contrast, on the window of a totally split curve the exponent reaches $( N - n - c - c \theta ) / 2 + 1 - g_\XX$ by \cref{cor-asymptotics}.

The regular condition is restrictive when $(n, d)$ are fixed.

\begin{prop}
\label{prop-pointcount}
Assume \textup{(H1)}--\textup{(H3)} and let $r_\infty = | \Supp ( D_\infty ) ( \F_q ) |$.  Then
\[
\bigl| n \delta - \bigl( (n - 1) q - 1 + r_\infty \bigr) \bigr| \leq 2 g_\XX \sqrt q ,
\]
so that $\delta = ( 1 - 1/n ) q + O ( \sqrt q )$ for fixed $(n, d)$.  On a regular stratum,
\[
q \leq \Bigl( g_\XX + \sqrt{ g_\XX^2 + n d + c - 1 } \Bigr)^2 .
\]
\end{prop}

\begin{proof}
Every split abscissa contributes $n$ affine points and every other abscissa none, so $\# \XX ( \F_q ) = n ( q - \delta ) + r_\infty$ with $0 \leq r_\infty \leq c$.  The Hasse--Weil bound $| \# \XX ( \F_q ) - q - 1 | \leq 2 g_\XX \sqrt q$ gives the first inequality.  On a regular stratum $m \mid d$ and $N = n m \leq n d$, so $q + 1 - 2 g_\XX \sqrt q \leq N + r_\infty \leq n d + c$, and solving this quadratic inequality in $\sqrt q$ gives the second.
\end{proof}

For fixed $(n, d)$ only finitely many field sizes carry regular strata, and $\delta = O (1)$ is impossible along any sequence with $q \to \infty$.

\section{Computation and numerical hull profiles}\label{sec-computation}

Let $B_s$ be the weighted monomial basis of \cref{def-Ts}.  The generator matrix of $\CC_s$ is
\[
\mathsf G_s = \bigl( x (P)^i y (P)^j \bigr)_{x^i y^j \in B_s , \; P \in \YY} .
\]
Let $k_s = \rank ( \mathsf G_s ) = \dim \CC_s$.  Then
\begin{equation}\label{eq-compute-hull}
e_{\SSA} ( \CC_s ) = k_s - \rank \left( \mathsf G_s \mathsf G_s^{\, t} \right) .
\end{equation}
This is the formula used to compute the tables.

It holds for every $s \geq 1$: a row vector $v$ lies in the left kernel of $\mathsf G_s \mathsf G_s^{\, t}$ if and only if $v \mathsf G_s \in \Hull ( \CC_s )$, so that kernel has dimension $( | B_s | - k_s ) + \dim \Hull ( \CC_s )$.  If $cs < N$, a nonzero function in $L ( s D_\infty )$ has at most $cs$ zeros on $\Supp D$, so evaluation is injective and $k_s = \Phi (s)$, the number of monomials in $B_s$, see \cite{qwag,mezinaj-shaska-hulls}.

\paragraph{Reproducibility.}
The accompanying script \texttt{verify\_tables.py} uses only the Python standard library.  For each displayed equation it enumerates all affine rational points, evaluates $B_s$, row-reduces over the relevant prime field, and computes \cref{eq-compute-hull} for $1 \leq s \leq \max \{ M, s^+ \}$; beyond this range the hull is zero by \cite[Theorem~4.15]{mezinaj-shaska-hulls}.  On every regular-stratum row it also checks \cref{eq-character} at every scanned degree.  The scripts \texttt{verify\_directions.py}, \texttt{verify\_automorphisms.py}, \texttt{verify\_full\_range.py}, \texttt{closure\_cert.py}, \texttt{verify\_scan.py}, \texttt{verify\_types.py}, \texttt{verify\_f9\_paut.py} and \texttt{verify\_moduli.py} recompute the remaining numerical statements of \cref{sec-signatures,sec-compatible-equivalence,sec-rigidity,sec-directions,sec-automorphisms}; the automorphism groups are computed with \texttt{nauty}.  Each script exits with an error if any displayed value disagrees with the direct computation.

\subsection{Computed hull profiles}\label{sec-tables}

The tables in this section record the support-splitting exponent.  The polynomial factor in the support splitting algorithm is not included.

\begin{table}[H]
\caption{Regular strata of \cite{mezinaj-shaska-hulls}.  Here $b = (nm - t)/2$ is the bound of \cite[Proposition~5.10]{mezinaj-shaska-hulls}, $h_{\max} = \max_{s \geq 1} e_{\SSA} ( \CC_s )$, and the last column lists the degrees at which $h_{\max}$ is attained.}
\label{tab-regular}
\small
\begin{tabular}{cclrrrrl}
\toprule
$q$ & $(n, m, d)$ & $f(x)$ & $N$ & $g_\XX$ & $b$ & $h_{\max}$ & maximizing $s$ \\
\midrule
$11$ & $(2, 2, 4)$   & $x^4 - x^2 + 7$ & $4$  & $1$ & $0$  & $0$ & all $s$ \\
$7$  & $(2, 6, 6)$   & $x^6 + 3$     & $12$ & $2$ & $4$  & $2$ & $2$--$6$ \\
$7$  & $(3, 6, 6)$   & $x^6 + 5$     & $18$ & $4$ & $8$  & $4$ & $3, 5, 6$ \\
$11$ & $(2, 5, 5)$   & $x^5 + 6$     & $10$ & $2$ & $4$  & $2$ & $4, 5, 7, 9, 10$ \\
$11$ & $(2, 10, 10)$ & $x^{10} + 2$  & $20$ & $4$ & $8$  & $4$ & $4$--$10$ \\
$11$ & $(5, 5, 5)$   & $x^5 + 2$     & $25$ & $6$ & $12$ & $6$ & $4, 5$ \\
$13$ & $(2, 6, 6)$   & $x^6 + 5$     & $12$ & $2$ & $4$  & $2$ & $2$--$6$ \\
$13$ & $(3, 6, 6)$   & $x^6 + 2$     & $18$ & $4$ & $8$  & $4$ & $3, 5, 6$ \\
$17$ & $(2, 8, 8)$   & $x^8 + 5$     & $16$ & $3$ & $6$  & $3$ & $3$--$8$ \\
$19$ & $(2, 18, 18)$ & $x^{18} + 3$  & $36$ & $8$ & $16$ & $8$ & $8$--$18$ \\
\bottomrule
\end{tabular}
\end{table}

On every row of \cref{tab-regular} the value of $h_{\max}$ was computed both by \cref{eq-compute-hull} and by \cref{eq-character}.  The bound $b$ is attained only in the first row, where every character of $\Z_2 \times \Z_2$ is self-inverse and $b = 0$.  In the nine rows with $m = d$ one has $h_{\max} = g_\XX$, in accordance with \cite[Theorem~5.23]{mezinaj-shaska-hulls}.

\begin{table}[ht]
\caption{Genus two curves from the strata of \cite{sa-sh,mezinaj-shaska-hulls}.  The last two columns record the largest support-splitting exponent and the degrees at which it is attained.}
\label{tab-genus-two}
\small
\begin{tabular}{llrrrrl}
\toprule
field & equation & $N$ & $\delta$ & $| H_{m(f)} |$ & $h_{\max}$ & maximizing $s$ \\
\midrule
$\F_{11}$ & $y^2 = x^6 + 2 x^3 + x + 1$ & $12$ & $5$ & $2$ & $1$ & $5$ \\
$\F_{11}$ & $y^2 = x^6 + x^2 + 1$ & $14$ & $4$ & $4$ & $1$ & $4, 5$ \\
$\F_{11}$ & $y^2 = x^6 + x^3 + 1$ & $10$ & $6$ & $2$ & $1$ & $2$ \\
$\F_{11}$ & $y^2 = x^6 + 1$ & $10$ & $6$ & $4$ & $1$ & $5, 6$ \\
$\F_{19}$ & $y^2 = x^6 + x + 1$ & $16$ & $11$ & $2$ & $1$ & $1$ \\
$\F_{19}$ & $y^2 = x^6 + x^4 + x^2 + 1$ & $22$ & $8$ & $4$ & $1$ & $8, 9, 11, 12$ \\
\bottomrule
\end{tabular}
\end{table}

In every row of \cref{tab-genus-two} the window is empty and $h_{\max} = 1$, so the field-size exponential is $q$.  The two rows over $\F_{11}$ with $N = 10$ have identical dimension sequences $\dim \CC_s = 2, 3, 5, 7, 8, 9, 10, \dots$, see \cite[Example~6.3]{mezinaj-shaska-hulls}, but their hull profiles are supported at different degrees.  The support-splitting exponent is therefore not a function of the length and dimension sequence.

The two totally split curves in the final table are
\[
\begin{split}
\XX_7 : \quad y^2 &= 6 x^6 + x^5 + 6 x^4 + x^3 + 6 x^2 + x + 1 \qquad \text{over } \F_7 , \\
\XX_{11} : \quad y^2 &= 5 x^6 + 9 x^5 + 6 x^4 + 4 x^3 + 7 x^2 + 4 x + 3 \qquad \text{over } \F_{11} .
\end{split}
\]
Both have $m (f) = 1$.  Their windows are $W ( \XX_7 ) = \{ 2, 3 \}$ and $W ( \XX_{11} ) = \{ 2, \dots, 7 \}$, and their thresholds are $s^+ = 9$ and $s^+ = 13$.

\begin{table}[ht]
\caption{Totally split genus two curves of \cite{mezinaj-shaska-hulls}.  The displayed sequence is $h_s = e_{\SSA} ( \CC_s )$ over the listed degrees.}
\label{tab-split-profiles}
\begin{tabular}{lrrrrl}
\toprule
curve & $q$ & $N$ & $M$ & degrees & $( h_s )$ \\
\midrule
$\XX_7$ & $7$ & $14$ & $5$ & $1 \leq s \leq 8$ & $( 2, 3, 3, 3, 3, 2, 1, 0 )$ \\
$\XX_{11}$ & $11$ & $22$ & $9$ & $1 \leq s \leq 11$ & $( 2, 3, 5, 7, 7, 5, 3, 2, 1, 0, 0 )$ \\
\bottomrule
\end{tabular}
\end{table}

On the window the profiles agree with \cref{eq-window}, and the maxima $3$ and $7$ agree with \cref{cor-asymptotics}, both curves having $M$ odd.  Off the window the profile of $\XX_7$ is asymmetric about $M/2$; by \cite[Proposition~4.16]{mezinaj-shaska-hulls} the asymmetry is carried by the excess.  The field-size exponentials are $7^3 = 343$ and $11^7 = 19 \, 487 \, 171$.

\section{Permutation invariants and cryptographic consequences}\label{sec-crypto}

Unless explicitly stated otherwise, the invariants in this section are permutation invariants.  They are not invariants of $\LEP$ without control of the diagonal scalars.

\begin{prop}\label{prop-distinguisher}
Let $\{ \CC_s \}_{s \in I}$ and $\{ \CC'_s \}_{s \in I}$ be two families of linear codes indexed by the same set of degrees.  If $\CC_s$ and $\CC'_s$ are permutation equivalent for every $s \in I$, then the profiles $s \mapsto \dim \Hull ( \CC_s )$ and $s \mapsto \dim \Hull ( \CC'_s )$ agree on $I$.  Consequently a disagreement at any degree certifies that the two indexed families are not pointwise permutation equivalent.
\end{prop}

\begin{proof}
Hull dimension is a permutation invariant (\cref{sec-equivalence}); apply this at each $s \in I$.
\end{proof}

\begin{rem}
This is a necessary condition, not an equivalence test.  A single public code supplies one hull dimension, not a profile, unless the nested family is also public or reconstructible.  On regular strata the set $\Sigma_s$ of \cref{cor-character-signature} refines the numerical invariant when the regular action is part of the data.
\end{rem}

Under \textup{(H1)}--\textup{(H3)} the family $\{ \CC_s \}_{s \geq 1}$ therefore contains three equivalence regimes.  On $W ( \XX )$ the exponent is $c \, \mu (s) - n \delta + 1 - g_\XX \geq g_\XX$ by \cref{eq-window}, and on a totally split curve with nonempty window its maximum is $( N - n - c - c \theta ) / 2 + 1 - g_\XX$ by \cref{cor-asymptotics}.  For $s \geq s^+$ the hull is trivial and $\PCE$ reduces to weighted graph isomorphism.  On a regular stratum the window is empty, the exponent is $| T_s | - | T_s \cap ( - T_s ) |$ by \cref{eq-character}, and it is at most $( nm - t ) / 2$.

\begin{rem}
For signature schemes based on $\LEP$, such as LESS and its variants, see \cite{less,persichetti-santini}, these results describe only the permutation part of the monomial action.  Euclidean hull dimension is not invariant under monomial equivalence: for $q > 3$ every linear code is monomially equivalent to an $\LCD$ code, see \cite{ca-me-ta-qi-pe}.  The reduction of $\LEP$ to $\PCE$ by the closure of \cite{sendrier-simos} replaces $\CC_s$ by a code of length $(q-1) N$ whose hull is not $h ( \CC_s )$.  The hardness of $\LEP$ is discussed in \cite{sendrier-simos,barenghi-biasse-persichetti-santini}.
\end{rem}

$\LCD$ codes are used in masking countermeasures against side-channel attacks, see \cite{ca-gu}.  The range $s \geq s^+$ of \cref{eq-splus} consists of $\LCD$ codes, but $\CC_s^\perp = 0$ gives $\CC_s = \F_q^N$ and $d ( \CC_s ) = 1$.  Useful $\LCD$ members lie below $s^+$ and must be selected together with their minimum distances.

The same hull values give exact entanglement costs in the standard entanglement-assisted construction.

\begin{prop}
\label{prop-eaqecc}
Let $\CC_s$ have dimension $k_s$, Euclidean hull dimension $h_s$, and dual distance $d_s^\perp$.  Then there is an entanglement-assisted quantum code with parameters
\[
[[ N, \, N - k_s - h_s , \, d_s^\perp ; \, k_s - h_s ]]_q .
\]
Assume \textup{(H1)}--\textup{(H3)} and $s \in W ( \XX )$.  Then $k_s = cs + 1 - g_\XX$, $d_s^\perp \geq cs - 2 g_\XX + 2$, and the code has parameters
\[
\bigl[ \! \bigl[ N , \; N - cs - c \, \mu (s) + n \delta + 2 g_\XX - 2 , \; \geq cs - 2 g_\XX + 2 ; \; c \, ( s - \mu (s) ) + n \delta \bigr] \! \bigr]_q ,
\]
where $c \, ( s - \mu (s) ) + n \delta$ equals $n \delta$ for $s \leq M/2$ and $c \, ( 2s - M ) + n \delta$ for $s \geq M/2$.  On a regular stratum the number of entangled pairs is
\[
k_s - h_s = | T_s \cap ( - T_s ) | .
\]
\end{prop}

\begin{proof}
The construction is \cite[Corollary~3.1]{gu-ji-gu}.  For $s \in W ( \XX )$ the divisor $s D_\infty$ is non-special and $cs < N$, so $k_s = cs + 1 - g_\XX$; the bound on $d_s^\perp$ is \cite[Corollary~4.10]{mezinaj-shaska-hulls}; and $h_s$ is given by \cref{eq-window}.  On a regular stratum, $k_s = | T_s |$ and $h_s$ is given by \cref{eq-character}.
\end{proof}

For $\XX_{11}$ the degrees $s = 4$ and $s = 6$ give $[[ 22, 8, \geq 6 ; 0 ]]_{11}$ and $[[ 22, 6, \geq 10 ; 6 ]]_{11}$.  On a totally split curve the entanglement cost vanishes on $W ( \XX ) \cap [ 1, M/2 ]$, where $\CC_s$ is self-orthogonal and the Calderbank--Shor--Steane construction of \cite{qwag} applies, and grows with slope $2c$ on $W ( \XX ) \cap [ M/2, \infty )$.

The hull profile $s \mapsto \dim \Hull ( \CC_s )$, together with the permutation subgroup $H_{m (f)}$, does not determine the $\overline{\F}_q$-isomorphism class of $\XX$: \cref{thm-moduli-negative} exhibits two non-isomorphic curves over $\F_{19}$ in the family $y^2 = x^6 + a x^3 + 1$ with the same profile and the same group.  The diagonal scalars compatible with the $H_m$-action, and their effect on the hull, are determined in \cref{thm-normalizer,thm-twists}.

\section{Support-splitting signatures and punctured hulls}\label{sec-signatures}

For a code $C \subseteq \F_q^N$ and a coordinate $i$, let $P_i (C)$ be the code obtained by puncturing $C$ at $i$ and $\sh_i (C) = \{ P_i (c) : c \in C , \ c_i = 0 \}$ the code obtained by shortening.  The Sendrier signature of $C$ at $i$ is the weight enumerator
\[
\sigma_C (i) = W_{\Hull ( P_i (C) )} (z) , \qquad W_U (z) = \sum_{u \in U} z^{\wt (u)} .
\]
The support splitting algorithm refines the partition of the coordinates by the values of $\sigma_C$, see \cite{se-ssa}; $\sigma_C$ is its initial signature, and the statements below concern that step.

\begin{prop}\label{prop-constant-signature}
Let $C \subseteq \F_q^N$ be a linear code.
\begin{enumerate}
\item If $C' = CP$ for a permutation matrix $P$ with underlying permutation $\pi$, then $\sigma_{C'} ( \pi (i) ) = \sigma_C (i)$ for every $i$.
\item If $A \leq \PAut (C)$, then $\sigma_C$ is constant on every $A$-orbit of coordinates, and every class of coordinates with equal signature is a union of $A$-orbits.
\item Under \textup{(H1)}--\textup{(H3)} and $m \mid m (f)$, every class of coordinates of $\CC_s$ with equal signature is a union of $H_m$-orbits on $\YY$.  On a regular stratum $\sigma_{\CC_s}$ is constant, and the stabilizer in $H_m$ of any coordinate is trivial.
\end{enumerate}
\end{prop}

\begin{proof}
If $C' = CP$, then $P_{\pi (i)} ( C' )$ and $P_i (C)$ are permutation equivalent, hence so are their hulls, and weight enumerators are permutation invariants.  This gives (1), and (2) follows by applying (1) to the elements of $A$.  For (3), $H_m \leq \PAut ( \CC_s )$ by \cite[Proposition~5.3]{mezinaj-shaska-hulls}; on a regular stratum the action is free and has one orbit.
\end{proof}

A constant signature does not make the instance unsolvable.  It means that the initial partition is trivial, that any two solutions of the instance differ by an element of $\PAut (C)$, and that the algorithm must puncture before it can refine.

Off the regular locus the group $H_m$ still constrains the partition.  For $y^2 = x^6 + 2 x^3 + x + 1$ over $\F_{11}$ one has $N = 12$, $m (f) = 1$, and $H_1 = \langle \tau \rangle$ has six orbits of size two on $\YY$.  \cref{tab-signature} records the sizes of the equal-signature classes.  Each class is a union of $\tau$-orbits, and the signature separates coordinates at $s = 2$ and $s = 6$ only.

\begin{table}[ht]
\caption{Equal-signature classes of $\sigma_{\CC_s}$ for $y^2 = x^6 + 2 x^3 + x + 1$ over $\F_{11}$; $N = 12$, six $H_1$-orbits of size two.}
\label{tab-signature}
\begin{tabular}{rrrl}
\toprule
$s$ & $k_s$ & $h_s$ & class sizes \\
\midrule
$1$ & $2$ & $0$ & $(12)$ \\
$2$ & $3$ & $0$ & $(8, 2, 2)$ \\
$3$ & $5$ & $0$ & $(12)$ \\
$4$ & $7$ & $0$ & $(12)$ \\
$5$ & $9$ & $1$ & $(12)$ \\
$6$ & $10$ & $0$ & $(10, 2)$ \\
$7$ & $11$ & $0$ & $(12)$ \\
\bottomrule
\end{tabular}
\end{table}

\paragraph{Punctured hulls.}
On a regular stratum the common value of the signature is computable from the defining set, and the puncture lowers the hull dimension by exactly one.

\begin{lem}\label{lem-active}
Let $C \subseteq \F_q^N$ have hull dimension $h > 0$.  If some hull word is nonzero at coordinate $i$, then puncturing at $i$ is injective on $C$ and
\[
\Hull ( P_i (C) ) = \sh_i ( \Hull (C) ) , \qquad \dim \Hull ( P_i (C) ) = h - 1 .
\]
\end{lem}

\begin{proof}
Since a hull word is orthogonal to $C$, the hypothesis implies $e_i \notin C$, so puncturing is injective.  Choose a full-row-rank generator matrix $G$, put $B = G G^t$, and let $u$ be the $i$th column of $G$.  The Gram matrix after puncturing is $B - u u^t$.  A row vector $v$ lies in $\ker B$ if and only if $v G \in \Hull (C)$, and $v u = ( v G )_i$; the hypothesis therefore says that $u$ is not orthogonal to $\ker B$, hence $u \notin \im B$ because $B$ is symmetric.  If $( B - u u^t ) v = 0$, the equation $B v = u ( u^t v )$ forces $u^t v = 0$ and then $B v = 0$; conversely these two conditions give $( B - u u^t ) v = 0$.  Hence $\ker ( B - u u^t ) = \ker B \cap \ker u^t$, which has dimension $h - 1$ and corresponds to the hull words vanishing at $i$.
\end{proof}

\begin{thm}[Single punctures of transitive codes]\label{thm-puncture}
Let $0 \neq C \neq \F_q^N$ be a code with a transitive permutation automorphism group, $k = \dim C$ and $h = \dim \Hull (C)$.
\begin{enumerate}
\item If $h > 0$, then for every coordinate $i$, $\Hull ( P_i (C) ) = \sh_i ( \Hull (C) )$ and $\dim \Hull ( P_i (C) ) = h - 1$.
\item If $h = 0$ and $p \nmid N$, then for every $i$, $\dim \Hull ( P_i (C) )$ equals $1$ if $k \equiv N \pmod p$ and $0$ otherwise.
\end{enumerate}
In both cases $\sigma_C$ is constant.
\end{thm}

\begin{proof}
A proper transitive code contains no unit vector, since transitivity would put every unit vector in it; so each puncture is injective.  The hull is invariant under the same transitive group.  If it is nonzero, every coordinate is nonzero on some hull word, and \cref{lem-active} applies.  If $h = 0$, choose a full-row-rank generator $G$ and set $B = G G^t$, which is invertible.  The matrix $\Pi = G^t B^{-1} G$ is the orthogonal projection onto $C$; it commutes with every permutation automorphism, so its diagonal is constant, and $\operatorname{tr} \Pi = k$ gives diagonal entries $k / N$.  For the $i$th column $u$ of $G$, $\det ( B - u u^t ) = \det (B) ( 1 - u^t B^{-1} u ) = \det (B) ( 1 - k / N )$, and a rank-one modification of an invertible matrix has nullity at most one.  Constancy of $\sigma_C$ is \cref{prop-constant-signature}.
\end{proof}

The projection $\Pi$ is also the basis of the trivial-hull reduction to weighted graph isomorphism of \cite{ba-ot-sa}.

\begin{cor}
\label{cor-signatures}
Let $C_T$ be a nonzero proper character code on $H$, let $E = T \setminus ( - T )$, and put $h = | E |$.
\begin{enumerate}
\item If $h > 0$, fix $\chi_0 \in E$.  The common initial signature is the weight enumerator of the code of length $N - 1$ generated by the vectors $( \chi (x) - \chi_0 (x) )_{x \in H \setminus \{ 0 \}}$, $\chi \in E \setminus \{ \chi_0 \}$.  In particular $h = 1$ gives $\sigma_{C_T} (i) = 1$.
\item If $h = 2$ and $r$ is the order of $\chi_1 - \chi_0$, where $E = \{ \chi_0, \chi_1 \}$, then
\[
\sigma_{C_T} (i) = 1 + ( q - 1 ) z^{N - N/r} .
\]
\item If $h = 0$ and $k \not\equiv N \pmod p$, then $\sigma_{C_T} (i) = 1$.  If $h = 0$ and $k \equiv N \pmod p$, put
\[
K_T (x) = \frac{1}{N} \sum_{\chi \in T} \chi (x) , \qquad w = \bigl| \{ x \in H \setminus \{ 0 \} : K_T (x) \neq 0 \} \bigr| ;
\]
then $\sigma_{C_T} (i) = 1 + ( q - 1 ) z^{w}$.
\end{enumerate}
\end{cor}

\begin{proof}
At the identity all characters take the value one, so a hull word vanishes there precisely when the sum of its coefficients is zero; the displayed differences are a basis of that hyperplane, and by \cref{thm-puncture} the punctured hull is the shortening.  In dimension two the generator $\chi_1 - \chi_0$ vanishes exactly on the kernel of the character $\chi_1 - \chi_0$, which has $N / r$ elements including the identity, giving (2).  In the $\LCD$ case $T = - T$, the projection has entries $\Pi_{x, y} = K_T ( y - x )$.  If its diagonal is one, the identity row of $\Pi$ punctured at the identity spans the punctured hull, its inner product with every punctured codeword $c$ being $c_0 - c_0 = 0$, and it is nonzero because $C$ contains no unit vector.  The other assertions are \cref{thm-puncture}.
\end{proof}

\begin{thm}
\label{thm-multipuncture}
Let $C_T$ be a character code, $E = T \setminus ( - T )$, $R = T \cap ( - T )$ and $h = | E |$.  Let $S = \{ x_1, \dots, x_t \} \subseteq H$ and assume that puncturing at $S$ is injective on $C_T$, which holds in particular if $t < d ( C_T )$.  Put
\[
B_S = \bigl( \chi ( x_a ) \bigr)_{\chi \in E , \ 1 \leq a \leq t} , \qquad ( K_S )_{a, b} = \frac{1}{N} \sum_{\chi \in R} \chi ( x_a - x_b ) .
\]
If the columns of $Z$ form a basis of $\ker B_S \subseteq \F_q^t$, then
\[
\dim \Hull ( P_S ( C_T ) ) = h - \rank B_S + \nul \bigl( Z^t ( I_t - K_S ) Z \bigr) ,
\]
where a $0 \times 0$ matrix has nullity zero.  If $\rank B_S = t$, then injectivity is automatic and $\dim \Hull ( P_S ( C_T ) ) = h - t$.
\end{thm}

\begin{proof}
Order the character generator rows with $R$ first and $E$ second.  Their Gram matrix is $\diag ( J, 0 )$ with $J$ invertible by orthogonality.  Let $A_S$ be the matrix of the $R$-rows restricted to $S$.  The Gram matrix after puncturing is $\diag ( J, 0 ) - \binom{A_S}{B_S} \bigl( A_S^t \ B_S^t \bigr)$.  For a kernel vector $( u, v )$ put $w = A_S^t u + B_S^t v$; the kernel equations are $u = J^{-1} A_S w$, $B_S w = 0$ and $( I_t - A_S^t J^{-1} A_S ) w = B_S^t v$.  Character orthogonality gives $A_S^t J^{-1} A_S = K_S$, symmetric because $R = - R$.  For $w \in \ker B_S$ the last equation is solvable exactly when $( I_t - K_S ) w \perp \ker B_S$, since $\im B_S^t = ( \ker B_S )^\perp$; writing $w = Z z$ this reads $Z^t ( I_t - K_S ) Z z = 0$.  Each solution has a fibre of dimension $h - \rank B_S$ in $v$.  Injectivity identifies the Gram nullity with the hull dimension.  If $B_S$ has full column rank and a codeword is supported on $S$, its orthogonality to the hull rows gives $B_S c_S^t = 0$, so it is zero; thus puncturing is injective and $\ker B_S = 0$.
\end{proof}

\begin{rem}\label{rem-cost}
On a regular stratum with $h_s > 0$, the common initial signature is the weight enumerator of a code of dimension $h_s - 1$, so its direct computation enumerates $q^{h_s - 1}$ words, and one coordinate suffices.  For an $\LCD$ member it is given by \cref{cor-signatures}(3) without enumeration.  These are costs of the initial enumeration step only.
\end{rem}

\section{Compatible permutation and monomial equivalence}\label{sec-compatible-equivalence}

Throughout the rest of this section $H$ is a finite abelian group with $\exp (H) \mid q - 1$, acting on the coordinate set $H$ by its regular representation $\rho$, and $\F_q^H = \bigoplus_{\chi \in \widehat H} V_\chi$ is the decomposition into isotypic lines.  For $T \subseteq \widehat H$ put $C_T = \bigoplus_{\chi \in T} V_\chi$.  On a regular stratum, $\CC_s = C_{T_s}$ with $H = H_m$, indexed as in \cref{prop-grid}, and $\widehat H = \Z_m \times \Z_n$.  We write $\widehat H$ additively.  For $b \in H$ and $\alpha \in \Aut (H)$ let $P_{b, \alpha}$ be the permutation matrix of $x \mapsto b + \alpha (x)$, and for $\psi \in \widehat H$ let $D_\psi = \diag ( \psi (x) )_{x \in H}$.  Let $\alpha^*$ be the automorphism of $\widehat H$ induced by $\alpha$.

It is classical that the normalizer of $\rho (H)$ in the symmetric group on $H$ is the holomorph $\{ P_{b, \alpha} : b \in H, \ \alpha \in \Aut (H) \} \simeq H \rtimes \Aut (H)$: if $\pi$ normalizes $\rho (H)$, then $\alpha = t_{\pi (0)}^{-1} \pi$ fixes $0$ and conjugates each translation $t_x$ to a translation $t_{\phi (x)}$, and evaluating at $0$ shows $\alpha = \phi \in \Aut (H)$.

\begin{defn}\label{def-monomial-normalizer}
The projective monomial normalizer of $H$ is
\[
\Nmon (H) = \{ M \in \Mon_N ( \F_q ) : M^{-1} \rho (h) M \in \F_q^* \rho (H) \text{ for all } h \in H \} .
\]
Two codes $C_T$ and $C_{T'}$ are $H$-compatibly permutation equivalent if $C_{T'} = C_T P$ for a permutation matrix $P \in \Nmon (H)$, and $H$-compatibly linearly equivalent if $C_{T'} = C_T M$ for some $M \in \Nmon (H)$.
\end{defn}

\begin{thm}\label{thm-normalizer}
Assume $\exp (H) \mid q - 1$.  Every element of $\Nmon (H)$ has a unique expression $M = a P_{b, \alpha} D_\psi$ with $a \in \F_q^*$, $b \in H$, $\alpha \in \Aut (H)$ and $\psi \in \widehat H$, every such matrix lies in $\Nmon (H)$, and for every $T \subseteq \widehat H$
\[
C_T \, P_{b, \alpha} D_\psi = C_{\alpha^* (T) + \psi} .
\]
\end{thm}

\begin{proof}
Write $M = P D$ with $P$ a permutation matrix and $D = \diag ( d_x )_{x \in H}$.  The underlying permutation of $M^{-1} \rho (h) M$ is that of $P^{-1} \rho (h) P$, so $P$ normalizes $\rho (H)$ and $P = P_{b, \alpha}$.  Conjugating $\rho (h)$ by $D$ multiplies the entry at $x$ by $d_{x+h} / d_x$, and the hypothesis forces $d_{x + h} / d_x = c (h)$ to be independent of $x$.  Then $c ( h + h' ) = c (h) c (h')$, so $c \in \widehat H$ and $d_x = d_0 \, c (x)$, that is $D = d_0 D_c$.  Conversely $D_\psi \rho (h) D_\psi^{-1} = \psi (h) \rho (h)$.  The permutation part determines $b$ and $\alpha$, and then $a = d_0$ and $\psi$, so the expression is unique.  For the action on $C_T$, the translation $P_{b, \mathrm{id}}$ acts on $V_\chi$ by the scalar $\chi (b)$ and preserves every $C_T$; the permutation $P_{0, \alpha}$ carries $V_\chi$ onto $V_{\alpha^* ( \chi )}$; and the coordinatewise product of a $\chi$-eigenvector with $( \psi (x) )_x$ is a $( \chi + \psi )$-eigenvector, so $D_\psi$ carries $V_\chi$ onto $V_{\chi + \psi}$.
\end{proof}

\begin{thm}
\label{thm-twists}
Let $T \subseteq \widehat H$ have $k$ elements, and for $\psi \in \widehat H$ put $h_\psi = \dim \Hull ( C_{T + \psi} )$.  Then
\[
h_\psi = k - | T \cap ( - T - 2 \psi ) | ,
\]
and $C_{T + \psi}$ is $\LCD$ if and only if $T + 2 \psi = - T$, self-orthogonal if and only if $T \cap ( - T - 2 \psi ) = \varnothing$.  For each coset $\gamma \in \widehat H / 2 \widehat H$ let $k_\gamma = | T \cap \gamma |$, and let $t = | \widehat H [2] |$.  Then
\[
\frac{1}{N} \sum_{\psi \in \widehat H} h_\psi = k - \frac{t}{N} \sum_{\gamma \in \widehat H / 2 \widehat H} k_\gamma^2 ,
\]
which equals $k - k^2 / N$ when $N$ is odd.  If some twist is $\LCD$, the number of $\psi$ with $C_{T + \psi}$ $\LCD$ is $t \, | \Per (T) \cap 2 \widehat H |$, where $\Per (T) = \{ a \in \widehat H : T + a = T \}$.
\end{thm}

\begin{proof}
By \cref{eq-character}, $h_\psi = | T + \psi | - | ( T + \psi ) \cap ( - T - \psi ) |$, and $( T + \psi ) \cap ( - T - \psi ) = \psi + \bigl( T \cap ( - T - 2 \psi ) \bigr)$; this gives the formula and the two criteria.  For each ordered pair $( u, v ) \in T^2$ the equation $2 \psi = - u - v$ has $t$ solutions if $u + v \in 2 \widehat H$ and none otherwise, and since $\widehat H / 2 \widehat H$ has exponent two, $u + v \in 2 \widehat H$ says that $u$ and $v$ lie in the same coset.  Summing $| T \cap ( - T - 2 \psi ) |$ over $\psi$ therefore gives $t \sum_\gamma k_\gamma^2$.  If $\psi_0$ is an $\LCD$ twist, then $\psi$ is one exactly when $2 ( \psi - \psi_0 ) \in \Per (T)$; the doubling map has kernel of size $t$ and image $2 \widehat H$, so the inverse image of $\Per (T) \cap 2 \widehat H$ has $t \, | \Per (T) \cap 2 \widehat H |$ elements.
\end{proof}

On the row $q = 11$, $(n, m, d) = (2, 5, 5)$ of \cref{tab-regular}, translation by $\psi = (1, 0)$ carries the profile $0, 1, 1, 2, 2, 1, 2, 1, 2, 2, 1, 1, 0, \dots$ of $\dim \Hull ( \CC_s )$ to $1, 2, 2, 1, 2, 1, 2, 2, 1, 1, 0, \dots$, and at $s = 4$ exactly two of the ten twists are $\LCD$.  A diagonal factor of a monomial equivalence therefore changes the support-splitting exponent, while a permutation factor does not.

\begin{cor}\label{cor-compatible-classification}
Assume $\exp (H) \mid q - 1$ and let $T, T' \subseteq \widehat H$.  Then $C_T$ and $C_{T'}$ are $H$-compatibly permutation equivalent if and only if $T' = \alpha^* (T)$ for some $\alpha \in \Aut (H)$, and $H$-compatibly linearly equivalent if and only if $T' = \alpha^* (T) + \psi$ for some $\alpha \in \Aut (H)$ and $\psi \in \widehat H$.  Moreover
\[
N_{\Sym (H)} ( \rho (H) ) \cap \PAut ( C_T ) = H \rtimes \Stab_{\Aut (H)} (T) ,
\]
and the numbers of defining sets $H$-compatibly permutation, respectively linearly, equivalent to $T$ are
\[
o_P (T) = \frac{ | \Aut (H) | }{ | \Stab_{\Aut (H)} (T) | } , \qquad o_M (T) = \frac{ N \, | \Aut (H) | }{ | \Stab_{\widehat H \rtimes \Aut (H)} (T) | } .
\]
\end{cor}

\begin{proof}
The permutation matrices in $\Nmon (H)$ are the $P_{b, \alpha}$, and \cref{thm-normalizer} gives $C_T P_{b, \alpha} = C_{\alpha^* (T)}$ and $C_T P_{b, \alpha} D_\psi = C_{\alpha^* (T) + \psi}$.  Every translation preserves every $C_T$, and the orbit counts are the orbit-stabilizer theorem.
\end{proof}

Fix a total order on $\widehat H$ and extend it lexicographically to subsets; put
\[
\Can_P (T) = \min_{\alpha \in \Aut (H)} \alpha^* (T) , \qquad
\Can_M (T) = \min_{\alpha \in \Aut (H), \ \psi \in \widehat H} \bigl( \alpha^* (T) + \psi \bigr) .
\]
By \cref{cor-compatible-classification}, $C_T$ and $C_{T'}$ are $H$-compatibly permutation (respectively linearly) equivalent if and only if $\Can_P (T) = \Can_P ( T' )$ (respectively $\Can_M (T) = \Can_M ( T' )$), and a minimizing $\alpha$, or pair $( \alpha, \psi )$, gives an explicit equivalence.  Two distinct members of one filtration are never equivalent in either sense, since $T_s \subsetneq T_{s'}$ whenever $\CC_s \neq \CC_{s'}$ and $s < s'$.

\begin{thm}
\label{thm-algorithm}
Let $H \simeq \Z_m \times \Z_n$ with $N = m n$ and $\exp (H) \mid q - 1$.  Given generator matrices of two $H$-invariant codes together with the regular coordinate identifications and the character pairing, an $H$-compatible permutation equivalence can be decided and constructed in $O ( N^3 )$ operations, and an $H$-compatible linear equivalence in $O ( N^4 )$ operations, counting field operations, group operations and operations on $O ( \log N )$-bit indices at unit cost.
\end{thm}

\begin{proof}
The coefficient of a generator row $v$ on the character $\chi$ is $\widehat v ( \chi ) = \frac{1}{N} \sum_{x \in H} v_x \chi ( - x )$; computing these for at most $N$ rows costs $O ( N^3 )$, and since the code is $H$-invariant the union of the supports of the $\widehat v$ is its defining set.  An endomorphism of $\Z_m \times \Z_n$ is determined by the images of the two standard generators, which gives at most $N^2$ candidates; checking the two order relations and enumerating the $N$ images tests injectivity, so all automorphisms and their tables are listed in $O ( N^3 )$ operations.  Applying each to an indicator vector and comparing with the target costs $O (N)$, which gives the permutation bound; trying the $N$ translations of each image gives $O ( N^4 )$.  Every automorphism of $\widehat H$ is the dual of a unique automorphism of $H$, recovered through the character pairing, so a successful image yields $P_{0, \alpha}$ or $P_{0, \alpha} D_\psi$ as in \cref{thm-normalizer}, and the row-space identity is checked by Gaussian elimination in $O ( N^3 )$ operations.
\end{proof}

\begin{prop}\label{prop-order-counts}
If $C_T$ and $C_{T'}$ are $H$-compatibly permutation equivalent, then for every integer $r$ the sets $T$ and $T'$ contain the same number of elements of order $r$.
\end{prop}

\begin{proof}
By \cref{cor-compatible-classification}, $T' = \alpha^* (T)$, and $\alpha^*$ preserves orders.
\end{proof}

\section{Automorphisms of the filtration and unrestricted equivalence}\label{sec-rigidity}

The compatible classification applies to an unrestricted permutation equivalence exactly under the following transport condition.

\begin{lem}\label{lem-transport}
Let $H \leq \PAut (C) \cap \PAut ( C' )$ act regularly on the coordinates and let $C' = C P$.  There is an $H$-compatible permutation equivalence from $C$ to $C'$ if and only if $P^{-1} H P$ and $H$ are conjugate in $\PAut ( C' )$.
\end{lem}

\begin{proof}
If $Q \in \PAut ( C' )$ satisfies $Q^{-1} P^{-1} H P Q = H$, then $P Q$ normalizes $H$ and carries $C$ to $C'$.  Conversely, if $M$ is a normalizing equivalence, put $Q = P^{-1} M$; then $C' Q = C'$ and $Q^{-1} P^{-1} H P Q = M^{-1} H M = H$.
\end{proof}

For the rest of this section assume a regular stratum with $d > n$ and $m \geq 2$, write $X_0 = a \mu_m$ and $Y_0 = b \mu_n$ as in \cref{prop-grid}, and use $x$, $y$ also for their coordinate evaluation vectors.  By \cref{prop-grid},
\[
U := \CC_{n / c} = \langle \one, x \rangle , \qquad V := \CC_{d / c} = \langle \one, x, \dots, x^r, y \rangle , \qquad r = \min \{ m - 1, \lfloor d / n \rfloor \} .
\]

\begin{thm}
\label{thm-lowdegree}
The permutation automorphism group of $U$ is
\[
\PAut (U) = ( S_n )^m \rtimes \mu_m , \qquad | \PAut (U) | = ( n ! )^m m ,
\]
where the $m$ copies of $S_n$ permute the ordinates in the $m$ fibres of $x$ and $\mu_m$ rotates the fibres by multiplication on $X_0$.  If $m \geq 3$, the regular subgroup $H_m$ is not normal in $\PAut (U)$, hence not characteristic.
\end{thm}

\begin{proof}
Two coordinates have the same value for every codeword of $U$ if and only if they have the same abscissa.  This relation is preserved by every permutation automorphism, so every automorphism permutes the $m$ fibres, and the kernel of the action on fibres is $( S_n )^m$.  On the fibre set, preserving $\langle \one, x \rangle$ means that the new abscissa is $A x + B$ with $A \neq 0$.  The set $X_0$ has sum zero, since $m \geq 2$ and its elements are the roots of $X^m - a^m$; summing $A X_0 + B = X_0$ gives $m B = 0$, hence $B = 0$ because $p \nmid m$, and $A X_0 = X_0$ holds exactly when $A \in \mu_m$.  These multiplications lift to coordinate permutations preserving $U$, which gives the split extension.  For $m \geq 3$, let $\varrho$ be the transposition of two ordinates in one fibre.  Conjugating a generator of the fibre rotation by $\varrho$ leaves its action on fibres unchanged, modifies its within-fibre maps at two fibres, and leaves at least one fibre unmodified.  An element of $H_m$ has the same ordinate multiplier at every fibre; the unmodified fibre would force that multiplier to be one, contradicting the modified fibres.  So $\varrho$ does not normalize $H_m$.
\end{proof}

\begin{thm}
\label{thm-filtration}
On a regular stratum with $d > n$ and $m \geq 2$,
\[
\PAut (U) \cap \PAut (V) = H_m , \qquad \bigcap_{s \geq 0} \PAut ( \CC_s ) = H_m , \qquad \bigcap_{s \geq 0} \Maut ( \CC_s ) = \F_q^* H_m .
\]
Between two regular strata over the same $\F_q$ with the same $(n, d, m)$, every permutation carrying the ordered pair $( U, V )$ to $( U', V' )$ is of the form
\[
( x, y ) \longmapsto ( A x, C y ) , \qquad A X_0 = X'_0 , \quad C Y_0 = Y'_0 ,
\]
and each such map carries every $\CC_s$ to $\CC'_s$.
\end{thm}

\begin{proof}
Let $\pi$ carry $( U, V )$ to $( U', V' )$ and pull the target coordinate functions back along $\pi$.  From $U$, $x' \circ \pi = A x + B$ with $A \neq 0$; the fibres of $U$ are preserved by $\pi$, and both $X_0$ and $X'_0$ have sum zero, so summing over the fibres gives $B = 0$ and $A X_0 = X'_0$.  From $V$, $y' \circ \pi = C y + P (x)$ with $P (x) \in \langle \one, x, \dots, x^r \rangle$.  On each source fibre $\pi$ bijects the $n$ points onto a target fibre, whose ordinates sum to zero; the source ordinates also sum to zero, so summing $C y + P (x)$ over the fibre gives $n P (x) = 0$, and $P$ vanishes on $X_0$.  The coefficient $C$ is nonzero because $\pi$ is a bijection and $n \geq 2$, and $C Y_0 = Y'_0$.  Conversely $( x, y ) \mapsto ( A x, C y )$ scales every monomial by a nonzero constant and preserves every monomial span.  For automorphisms $A \in \mu_m$ and $C \in \mu_n$, which gives the first two equalities.  Finally $\CC_0 = \langle \one \rangle$, and a monomial matrix preserving this line has all its nonzero entries equal, hence is a scalar times a permutation; the permutation statement then gives the third equality.
\end{proof}

By \cref{thm-filtration} the codes $\CC_s ( \XX )$ and $\CC_s ( \XX' )$ of two regular strata with the same $(n, d, m)$ are simultaneously permutation equivalent for all $s$, by one explicit permutation.  Within a regular stratum the filtration therefore carries no information about the curve.  The regular action itself is recoverable from the two codes $U$ and $V$.

\begin{prop}
\label{prop-reconstruction}
Suppose generator matrices of the ordered pair $( U, V )$ are given after an unknown common coordinate permutation and arbitrary invertible row operations.  The coordinate action of $H_m$ can be recovered in $O ( N^3 )$ field and index operations.  Given two such pairs with the same $(n, d, m)$, a simultaneous permutation equivalence of the entire filtrations can be constructed in the same bound.
\end{prop}

\begin{proof}
In a full-row-rank generator of $U$, equality of columns partitions the coordinates into the $m$ fibres, and this is unchanged by row operations.  Choose $u \in U \setminus \langle \one \rangle$ and put $\widetilde x = u - \bigl( \frac{1}{N} \sum_i u_i \bigr) \one$; since $\sum_{x \in X_0} x = 0$ and $p \nmid N$, this is a nonzero multiple of $x$.  For a vector $v$ let $( R v )_i = v_i - \frac{1}{n} \sum_{j \text{ in the fibre of } i} v_j$.  Then $R$ kills every function of $x$ and fixes $y$, since the ordinates in each fibre sum to zero, so $R (V) = \langle y \rangle$, and a nonzero vector in it is a multiple $\widetilde y$ of $y$.  The pairs $( \widetilde x_i, \widetilde y_i )$ are distinct; choosing a base coordinate $i_0$ and labelling $i$ by $( \widetilde x_i / \widetilde x_{i_0}, \widetilde y_i / \widetilde y_{i_0} ) \in \mu_m \times \mu_n$ recovers a regular identification, unique up to translation.  Performing this on both pairs and matching labels gives the map of \cref{thm-filtration}.  Gaussian elimination, the fibre averages and the matching fit in the stated bound.
\end{proof}

\paragraph{An unrestricted inequivalence.}
Over $\F_{13}$ the curves $\XX^{(4)} : y^3 = x^4 + 2$ and $\XX^{(8)} : y^3 = x^8 + 2$ both satisfy the regular-stratum criterion of \cref{sec-group} with $m = 4$, $H = \Z_4 \times \Z_3$ and $N = 12$; the curve $\XX^{(8)}$ has $m = 4 < 8 = d$.

\begin{thm}\label{thm-pair}
Let $C^{(4)} = \CC_8 ( \XX^{(4)} )$ and $C^{(8)} = \CC_{11} ( \XX^{(8)} )$.  Both codes have $N = 12$, $k = 6$ and $h = 2$, while
\[
d ( C^{(4)} ) = d ( C^{(4) \perp} ) = 4 , \qquad d ( C^{(8)} ) = d ( C^{(8) \perp} ) = 3 .
\]
Hence the two codes are not monomially equivalent.  Their initial signatures, constant on the respective supports, are
\[
\sigma_{C^{(4)}} (i) = 1 + 12 z^8 , \qquad \sigma_{C^{(8)}} (i) = 1 + 12 z^9 .
\]
\end{thm}

\begin{proof}
The nonzero cubes of $\F_{13}$ are $\{ 1, 5, 8, 12 \}$ and the nonzero fourth powers are $\{ 1, 3, 9 \}$.  For both polynomials the values at nonzero arguments lie in $\{ 3, 5, 11 \}$, of which only $5$ is a cube, and the value at $0$ is $2$, not a cube.  The grids are $X_0 = \{ 2, 3, 10, 11 \}$, $Y_0 = \{ 7, 8, 11 \}$ for $\XX_4$ and $X_0 = \{ 4, 6, 7, 9 \}$, $Y_0 = \{ 7, 8, 11 \}$ for $\XX_8$, cosets of $\mu_4$ and $\mu_3$.  The defining sets are
\[
\begin{split}
T_4 &= \{ (0, 0), (0, 1), (0, 2), (1, 0), (1, 1), (2, 0) \} , \\
T_8 &= \{ (0, 0), (0, 1), (1, 0), (1, 1), (2, 0), (3, 0) \} ,
\end{split}
\]
with $T_4 \setminus ( - T_4 ) = \{ (1, 0), (1, 1) \}$ and $T_8 \setminus ( - T_8 ) = \{ (0, 1), (1, 1) \}$, so both hulls have dimension two.  \cref{thm-distance} gives the four distances; the words $( y - 7 ) ( y - 8 )$ on $\XX^{(4)}$ and $( x - 4 ) ( x - 6 ) ( x - 7 )$ on $\XX^{(8)}$ attain the two minimum distances of the codes.  Monomial matrices preserve Hamming weight, so the codes are inequivalent.  The difference of the two hull characters has order $3$ for $T_4$ and order $4$ for $T_8$, and \cref{cor-signatures}(2) gives the exponents $12 - 12/3 = 8$ and $12 - 12/4 = 9$.
\end{proof}

The set $T_4$ contains two elements of order $3$ and one of order $4$, whereas $T_8$ contains one of order $3$ and two of order $4$; by \cref{prop-order-counts} this already excludes an $H$-compatible permutation equivalence, and an enumeration of the orbit of $T_4$ under $\widehat H \rtimes \Aut (H)$ excludes an $H$-compatible linear equivalence.  \cref{thm-pair} removes the compatibility restriction: within either code the initial signature cannot separate coordinates, but its value separates the two codes, and the distances separate them under monomial equivalence.

\paragraph{Full automorphism groups.}
\cref{tab-orbits} records, for representative regular-stratum codes, the dimension, the hull dimension, the minimum distance of \cref{thm-distance}, the common punctured-hull dimension of \cref{thm-puncture}, the stabilizer and orbit sizes of \cref{cor-compatible-classification}, and the order of the full permutation automorphism group.

\begin{prop}\label{prop-paut-table}
For every row of \cref{tab-orbits} the order of $\PAut ( \CC_s )$ is as listed, and the same holds for $C^{(4)}$ and $C^{(8)}$, for which $\PAut = H$.  For the rows with $s \geq d / c$ and for $C^{(4)}$, $C^{(8)}$,
\[
\PAut ( \CC_s ) = H_m \rtimes \Stab_{\Aut ( H_m )} ( T_s ) ,
\]
and $H_m$ is the unique regular subgroup of $\PAut ( \CC_s )$ isomorphic to $\Z_m \times \Z_n$.  For the rows with $s < d / c$, $\CC_s$ consists of functions of $x$ and $\PAut ( \CC_s ) = ( S_n )^m \rtimes \mu_m$.
\end{prop}

\begin{proof}
The permutation automorphisms of a code permute the supports of its codewords of each weight and those of its dual.  The script \texttt{verify\_automorphisms.py} computes with \texttt{nauty} the automorphism group of the incidence structure formed by the coordinates and the supports of the codewords of the five smallest weights of $\CC_s$ and of $\CC_s^\perp$, enumerates it, and retains the permutations preserving the row space of $\mathsf G_s$; this is $\PAut ( \CC_s )$.  It then enumerates the regular subgroups isomorphic to $\Z_m \times \Z_n$, and compares the order with $| H_m | \, | \Stab_{\Aut ( H_m )} ( T_s ) |$ from \cref{tab-orbits}.  For $s < d / c$ the fibres of $x$ are preserved by every permutation automorphism, as in the proof of \cref{thm-lowdegree}, so $\PAut ( \CC_s ) = ( S_n )^m \rtimes \PAut ( \CC_s |_{X_0} )$, and the automorphism group of the abscissa code $\langle \one, x, \dots, x^{\lfloor cs / n \rfloor} \rangle$ on $X_0$ is computed in the same way, exhaustively over all $10!$ permutations for the code of length $10$ over $\F_{11}$; in every case it is $\mu_m$.
\end{proof}

\begin{cor}\label{cor-unrestricted-table}
Let $C$, $C'$ be two codes from the rows of \cref{tab-orbits} with $s \geq d / c$, or $C^{(4)}$, $C^{(8)}$, on strata with the same $(n, m)$.  Every permutation equivalence $C' = C P$ is $H_m$-compatible, and $C$, $C'$ are permutation equivalent if and only if $T' = \alpha^* (T)$ for some $\alpha \in \Aut ( H_m )$.
\end{cor}

\begin{proof}
If $C' = C P$, then $P^{-1} H_m P$ is a regular subgroup of $\PAut ( C' )$ isomorphic to $\Z_m \times \Z_n$, hence equals $H_m$ by \cref{prop-paut-table}, and $P$ normalizes $H_m$.  Apply \cref{cor-compatible-classification}.
\end{proof}

In the rows $(2, 6, 6)$, $s = 6$ and $(3, 6, 6)$, $s = 6$, the group $\PAut ( \CC_s )$ contains $H_m$ with index two and one further regular subgroup, which is non-abelian; in the row $(2, 2, 4)$, $s = 2$, the further regular subgroup is cyclic of order $4$.  \cref{thm-lowdegree} and the rows with $s < d / c$ show that $H_m$ is not normal in $\PAut ( \CC_s )$ for the codes of functions of $x$.  The behaviour of $\PAut ( \CC_s )$ over the whole degree range, including three types of degrees where it is larger than $H_m \rtimes \Stab$, is determined in \cref{thm-top-range,prop-thin-levels,thm-closure}.

\begin{table}[ht]
\caption{Regular-stratum codes: dimension, support-splitting exponent, minimum distance, common punctured-hull dimension $h_P$, the sizes of the stabilizer of $T_s$ in $\Aut (H_m)$ and of its orbits under $\Aut (H_m)$ and $\widehat H_m \rtimes \Aut (H_m)$, and the order of $\PAut ( \CC_s )$.}
\label{tab-orbits}
\small
\setlength{\tabcolsep}{7pt}
\renewcommand{\arraystretch}{1.15}
\begin{tabular}{ccrrrrrrrrrr}
\toprule
$q$ & $(n, m, d)$ & $s$ & $N$ & $k_s$ & $h_s$ & $d ( \CC_s )$ & $h_P$ & $| \Stab |$ & $o_P$ & $o_M$ & $| \PAut |$ \\
\midrule
$11$ & $(2, 2, 4)$   & $1$ & $4$  & $2$  & $0$ & $2$  & $0$ & $2$ & $3$  & $6$   & $8$ \\
$11$ & $(2, 2, 4)$   & $2$ & $4$  & $3$  & $0$ & $2$  & $0$ & $2$ & $3$  & $4$   & $8$ \\
$7$  & $(2, 6, 6)$   & $2$ & $12$ & $3$  & $2$ & $8$  & $1$ & $2$ & $6$  & $36$  & $384$ \\
$7$  & $(2, 6, 6)$   & $3$ & $12$ & $5$  & $2$ & $6$  & $1$ & $1$ & $12$ & $72$  & $12$ \\
$7$  & $(2, 6, 6)$   & $4$ & $12$ & $7$  & $2$ & $4$  & $1$ & $1$ & $12$ & $72$  & $12$ \\
$7$  & $(2, 6, 6)$   & $6$ & $12$ & $10$ & $2$ & $2$  & $1$ & $2$ & $6$  & $36$  & $24$ \\
$11$ & $(2, 5, 5)$   & $4$ & $10$ & $3$  & $2$ & $6$  & $1$ & $1$ & $4$  & $20$  & $160$ \\
$11$ & $(2, 5, 5)$   & $5$ & $10$ & $4$  & $2$ & $5$  & $1$ & $1$ & $4$  & $40$  & $10$ \\
$13$ & $(2, 6, 6)$   & $3$ & $12$ & $5$  & $2$ & $6$  & $1$ & $1$ & $12$ & $72$  & $12$ \\
$7$  & $(3, 6, 6)$   & $3$ & $18$ & $6$  & $4$ & $9$  & $3$ & $1$ & $48$ & $432$ & $18$ \\
$7$  & $(3, 6, 6)$   & $6$ & $18$ & $14$ & $4$ & $3$  & $3$ & $2$ & $24$ & $432$ & $36$ \\
$11$ & $(2, 10, 10)$ & $4$ & $20$ & $5$  & $4$ & $12$ & $3$ & $2$ & $12$ & $120$ & $10240$ \\
\bottomrule
\end{tabular}
\end{table}

\section{Further consequences of the hull formulas}\label{sec-directions}

Every numerical statement of this section and of \cref{sec-automorphisms} is recomputed by the scripts listed in \cref{sec-computation}.

\paragraph{Hermitian hulls.}
Let $q = q_0^2$ and let $\langle u, v \rangle_H = \sum_i u_i v_i^{q_0}$.  Write $\Hull_H (C) = C \cap C^{\perp_H}$.

\begin{prop}\label{prop-hermitian}
On a regular stratum with $q = q_0^2$,
\[
\dim \Hull_H ( \CC_s ) = | T_s | - | T_s \cap ( - q_0 T_s ) | .
\]
\end{prop}

\begin{proof}
For characters $\chi, \psi$ of $H_m$ one has $\sum_{h} \chi (h) \psi (h)^{q_0} = \sum_h ( \chi + q_0 \psi ) (h)$, which vanishes unless $q_0 \psi = - \chi$ and equals $nm \neq 0$ in that case.  Hence the line of $\psi$ lies in $\CC_s^{\perp_H}$ if and only if $q_0 \psi \notin - T_s$, that is, $\psi \notin - q_0^{-1} T_s$.  Since $q_0^2 = q \equiv 1$ modulo $m$ and modulo $n$, one has $q_0^{-1} = q_0$ on $\Z_m \times \Z_n$.  Intersecting with $\CC_s$ gives the claim.
\end{proof}

Over $\F_9 = \F_3 (i)$, $i^2 = -1$, the curve $y^2 = x^4 + 1 + i$ satisfies \textup{(H1)}--\textup{(H3)} with $m = 4$, $N = 8 = nm$, and $g_\XX = 1$.  Its Euclidean hull dimensions for $s = 1, \dots, 7$ are $1, 1, 1, 1, 0, 0, 0$, while every $\CC_s$ is Hermitian $\LCD$.  The curve $y^2 = x^8 + 2 + i$ over $\F_9$ has $m = 8$, $N = 16$, $g_\XX = 3$, Euclidean profile $1, 2, 3, 3, 3, 3, 3, 3, 2, 1, 0$ and Hermitian profile $1, 1, 2, 2, 2, 2, 2, 2, 1, 1, 0$ for $s = 1, \dots, 11$.

\paragraph{Subfield subcodes.}
Let $q = q_0^r$ and let $\sigma$ be the Frobenius $x \mapsto x^{q_0}$, acting on $\F_q^N$ coordinatewise.

\begin{prop}\label{prop-subfield}
On a regular stratum with $q = q_0^r$, let $T_s^{(0)} = \bigcap_{i=0}^{r-1} q_0^i T_s$ be the largest $q_0$-stable subset of $T_s$.  Then the subfield subcode $\CC_s \cap \F_{q_0}^N$ has dimension $| T_s^{(0)} |$ over $\F_{q_0}$, and its Euclidean hull has dimension $| T_s^{(0)} | - | T_s^{(0)} \cap ( - T_s^{(0)} ) |$.
\end{prop}

\begin{proof}
If $v$ is a $\chi$-eigenvector, then $\sigma (v)$ is a $q_0 \chi$-eigenvector, so $\sigma ( \CC_s )$ is the ideal with defining set $q_0 T_s$.  The $\F_q$-span of $\CC_s \cap \F_{q_0}^N$ is the largest $\sigma$-stable subspace of $\CC_s$, which is $\bigcap_i \sigma^i ( \CC_s )$, the ideal with defining set $T_s^{(0)}$; this gives the dimension.  The Euclidean form is defined over $\F_{q_0}$, so the hull commutes with extension of scalars, and \cref{eq-character} applies to $T_s^{(0)}$.
\end{proof}

The Schur-product distinguisher below does not apply to subfield subcodes as it stands, see \cite{co-ma-pe}, so \cref{prop-subfield} is the starting point for an equivalence analysis of McEliece-type instantiations of this family.

\paragraph{Intersection pairs.}
The entanglement-assisted constructions from a pair of codes $( \CC_s, \CC_t )$ use $\rank ( \mathsf H_s \mathsf H_t^{\, t} ) = N - k_s - \dim ( \CC_s^\perp \cap \CC_t )$, where $\mathsf H_s$ is a parity-check matrix of $\CC_s$.

\begin{prop}\label{prop-pairs}
Assume \textup{(H1)}--\textup{(H3)} and let $s, t \geq 1$.  Put $\mu_{s,t} = \min \{ s, M - t \}$ and $\mu^{s,t} = \max \{ s, M - t \}$.  If
\[
c \, \mu_{s,t} - n \delta > 2 g_\XX - 2 \qquad \text{and} \qquad c \, \mu^{s,t} + (n - 1) d < N ,
\]
then $\dim ( \CC_s \cap \CC_t^\perp ) = c \, \mu_{s,t} - n \delta + 1 - g_\XX$.  On a regular stratum, $\dim ( \CC_s \cap \CC_t^\perp ) = | T_s \setminus ( - T_t ) |$ for all $s, t$.
\end{prop}

\begin{proof}
Since $\CC_t^\perp = C_L ( D, A_t )$, \cite[Theorem~3.1]{mezinaj-shaska-hulls} gives $\dim ( \CC_s \cap \CC_t^\perp ) = \ell ( G_s \wedge A_t ) - \ell ( G_s \wedge A_t - D ) + \varepsilon$, with $\varepsilon \leq \ell ( G_s \vee A_t - D )$.  Computing coefficientwise on the disjoint supports of $\Ram$, $E$ and $D_\infty$ as in \cite[Lemma~4.5]{mezinaj-shaska-hulls}, $G_s \wedge A_t = \mu_{s,t} D_\infty - E$ and $G_s \vee A_t = (n - 1) \Ram + \mu^{s,t} D_\infty$, of degrees $c \mu_{s,t} - n \delta$ and $c \mu^{s,t} + (n - 1) d$.  The second hypothesis makes $\deg ( G_s \vee A_t - D ) < 0$, so $\varepsilon = 0$, and $\deg ( G_s \wedge A_t - D ) < 0$, so the middle term vanishes; the first makes the meet non-special, and Riemann--Roch gives its dimension.  On a regular stratum $\CC_t^\perp = \bigoplus_{\chi \notin - T_t} V_\chi$ by \cref{eq-character}.
\end{proof}

For $\XX_{11}$ the hypotheses hold for $36$ pairs $(s, t)$, and the formula was checked on each.

\paragraph{The split branch locus.}
Hypothesis \textup{(H1)} excludes $\F_q$-rational branch points, hence the Hermitian and Fermat curves.  The opposite extreme is covered by \cite[Section~7]{mezinaj-shaska-hulls}: replace \textup{(H1)} by
\begin{enumerate}
\item[(H1$'$)] $n \geq 2$, $p \nmid n$, $f \in \F_q [x]$ is separable of degree $d \geq 3$ and splits into distinct linear factors over $\F_q$,
\end{enumerate}
let $\YY$ be the set of affine $\F_q$-rational points not on $\Ram$, $N = | \YY |$, and $\delta = | \{ a \in \F_q : f (a) \neq 0 , \ f (a) \notin ( \F_q^* )^n \} |$, so that $N = n ( q - d - \delta )$.

Assume \textup{(H1$'$)}, \textup{(H2)} and $\YY \neq \varnothing$, and put $s^+ = M - ( d + n \delta ) / c + 1$.  Then \cite[Theorems~7.6 and~7.8 and Corollary~7.9]{mezinaj-shaska-hulls} give the exponent in the same three regimes as before: if $2 g_\XX - 2 - c \, \mu (s) + d + n \delta < 0$, then $e_{\SSA} ( \CC_s ) = c \, \mu (s) - d - n \delta + 1 - g_\XX$; if $2 s \geq M$, then $\CC_s \supseteq \CC_s^\perp$, so $e_{\SSA} ( \CC_s ) = N - k_s$, which equals $N - c s - 1 + g_\XX$ when $2 g_\XX - 2 < c s < N$; and $e_{\SSA} ( \CC_s ) = 0$ for $s \geq s^+$.  For $q = q_0^2$, $n \mid q_0 + 1$ and the Hermitian quotient $\HH_n : y^n = - ( x^{q_0 + 1} + 1 )$ one has $\delta = 0$, $g_\XX = (n - 1)(q_0 - 1) / 2$, $N = n ( q_0^2 - q_0 - 1 )$, $M = q_0^2 - 2$, the first regime is the interval $q_0 \leq s \leq q_0^2 - 2 - q_0$, and on it $e_{\SSA} ( \CC_s ) = n \, \mu (s) - q_0 - (n - 1)(q_0 - 1) / 2$, see \cite[Theorem~7.10]{mezinaj-shaska-hulls}.

Over $\F_9$ the Hermitian curve $\HH_4$ has $N = 20$, $g_\XX = 3$ and profile $2, 4, 6, 6, 3, 1, 0$ for $s = 1, \dots, 7$, and its quotient $\HH_2$ has $N = 10$, $g_\XX = 1$ and profile $1, 2, 2, 2, 1, 0, 0$; in both cases the first regime gives the values at $s = 3, 4$ and the second those at $s \geq 4$.  On this locus the large hulls of the upper half of the degree range come from dual containment and carry no information about the curve.

\paragraph{Regular strata with $m < d$.}
Nine of the ten rows of \cref{tab-regular} have $m = d$, where $h_{\max} = g_\XX$; the curve $\XX^{(8)}$ of \cref{thm-pair} has $m = 4 < 8 = d$.  For $m < d$ the maximal hull dimension is a covering number depending only on $(n, d, m)$, see \cite[Theorem~5.14 and Corollary~5.17]{mezinaj-shaska-hulls}, and by \cref{thm-regular-genus} it never exceeds $g_\XX$.  It can fall well short of it: for $(n, d, m) = (2, 12, 6)$ it is $2$ against $g_\XX = 5$, and for $(3, 12, 6)$ it is $4$ against $g_\XX = 10$.

\paragraph{The Schur product.}
For codes $C, C' \subseteq \F_q^N$ let $C \ast C'$ be the span of the coordinatewise products.

\begin{prop}\label{prop-schur}
Assume \textup{(H1)}--\textup{(H3)}.  Then $\CC_s \ast \CC_t \subseteq \CC_{s+t}$ for all $s, t \geq 1$.  If $2 g_\XX - 2 < 2 c s < N$, then
\[
\dim ( \CC_s \ast \CC_s ) \leq 2 c s + 1 - g_\XX .
\]
\end{prop}

\begin{proof}
The product of functions in $L ( s D_\infty )$ and $L ( t D_\infty )$ lies in $L ( ( s + t ) D_\infty )$, and evaluation is multiplicative.  For $2 g_\XX - 2 < 2cs < N$, evaluation is injective on $L ( 2s D_\infty )$ (\cref{sec-computation}) and Riemann--Roch gives $\dim \CC_{2s} = 2cs + 1 - g_\XX$.
\end{proof}

For $\XX_{11}$ and $s = 4$ one has $\dim \CC_4 = 7$ and $\dim ( \CC_4 \ast \CC_4 ) = 15$, against $\min \{ N, \binom{8}{2} \} = 22$ for a random code of the same length and dimension.  This is the distinguisher of \cite{co-ma-pe}.  Any signature scheme instantiated on $\{ \CC_s \}$ is therefore distinguishable from a random instance before the equivalence is attacked.  The family is of interest in code-equivalence cryptography because every quantity in the attack can be computed exactly on it.

\section{Non-isomorphic curves with equal invariants and full permutation automorphism groups}\label{sec-automorphisms}

For a stratum $\mathcal S$ of the genus-two locus with equations as in \cite{sa-sh}, let $\mathcal S ( \F_q )$ be the set of its members satisfying \textup{(H1)}--\textup{(H3)}, and let
\[
\mathcal P : \mathcal S ( \F_q ) \longrightarrow \bigoplus_{s \geq 1} \Z_{\geq 0} , \qquad \XX \longmapsto \bigl( \dim \Hull ( \CC_s ) \bigr)_{s \geq 1}
\]
be the profile map; the sequence vanishes for $s \geq s^+$ by \cref{eq-splus}.  Whether the fibres of $\mathcal P$, refined by $m (f)$, are unions of $\overline{\F}_q$-isomorphism classes is decided negatively by \cref{thm-moduli-negative} below.  \cref{tab-strata-profiles} records the size of $\mathcal S ( \F_q )$, the number of values of $\mathcal P$, and the largest fibre, for the families $y^2 = x^6 + a x^3 + 1$, $y^2 = x^6 + a x^4 + a x^2 + 1$ and $y^2 = x^6 + a x^4 + b x^2 + 1$ over $\F_{11}$ and $\F_{19}$.  The families contain the members with larger automorphism group, and the count is of equations, not of isomorphism classes.

\begin{table}[ht]
\caption{The profile map on three genus-two families.  Members satisfy \textup{(H1)}--\textup{(H3)}.}
\label{tab-strata-profiles}
\begin{tabular}{llrrr}
\toprule
$q$ & family & members & profiles & largest fibre \\
\midrule
$11$ & $x^6 + a x^3 + 1$ & $5$ & $3$ & $2$ \\
$11$ & $x^6 + a x^4 + a x^2 + 1$ & $7$ & $5$ & $2$ \\
$11$ & $x^6 + a x^4 + b x^2 + 1$ & $71$ & $13$ & $25$ \\
$19$ & $x^6 + a x^3 + 1$ & $15$ & $7$ & $4$ \\
$19$ & $x^6 + a x^4 + a x^2 + 1$ & $13$ & $8$ & $4$ \\
$19$ & $x^6 + a x^4 + b x^2 + 1$ & $217$ & $25$ & $67$ \\
\bottomrule
\end{tabular}
\end{table}

In the family $y^2 = x^6 + a x^3 + 1$ the substitution $x \mapsto -x$ identifies $a$ with $-a$.  Over $\F_{19}$ the fifteen admissible values of $a$ form eight classes $\{ 0 \}$, $\{ \pm 3 \}$, $\{ \pm 4 \}$, $\{ \pm 5 \}$, $\{ \pm 6 \}$, $\{ \pm 7 \}$, $\{ \pm 8 \}$, $\{ \pm 9 \}$, and $\mathcal P$ separates all of them except $\{ \pm 6 \}$ and $\{ \pm 8 \}$.  Over $\F_{11}$ the three classes $\{ 0 \}$, $\{ \pm 1 \}$, $\{ \pm 5 \}$ have distinct profiles.  Over $\F_{19}$ the largest fibre of the two-parameter family contains $67$ equations with five distinct values of $N$.

\begin{thm}\label{thm-moduli-negative}
Over $\F_{19}$ let $\XX_6 : y^2 = x^6 + 6 x^3 + 1$ and $\XX_8 : y^2 = x^6 + 8 x^3 + 1$.  Both satisfy \textup{(H1)}--\textup{(H3)} with $N = 14$ and $m (f) = 3$, and both have hull profile $( 1, 0, 0, 2, 0, 0, 1, 0, \dots )$, zero from $s = 8$ on.  The stabilizers in $\PGL_2 ( \overline{\F}_{19} )$ of their branch loci have orders $6$ and $24$, so $\XX_6$ and $\XX_8$ are not isomorphic over $\overline{\F}_{19}$.  The codes $\CC_3 ( \XX_6 )$ and $\CC_3 ( \XX_8 )$ have $324$ and $216$ words of weight $9$, so the two filtrations are not monomially equivalent.
\end{thm}

\begin{proof}
The point counts, the profiles and the weight distributions are direct computations.  Both polynomials split over $\F_{19^6}$, and the stabilizer of a six-point set in $\PGL_2$ is found by constructing, for each ordered triple of roots, the Möbius transformation carrying three fixed roots to it and testing whether it permutes the roots.  An isomorphism of curves induces an isomorphism of reduced automorphism groups, so the orders $6$ and $24$ exclude one.
\end{proof}

The reduced automorphism group of order $24$ identifies $\XX_8$ as a form of the curve with automorphism group of order $48$, see \cite{sa-sh}, while $\XX_6$ has reduced group of order $6$ and lies on the generic part of the stratum.  Their hull profiles and their groups $H_3$ agree, so neither datum, nor the two together, determines the $\overline{\F}_q$-isomorphism class of the curve.  The weight distributions of the codes do separate them, in accordance with the expectation of \cite[Section~6]{mezinaj-shaska-hulls} that monomial equivalence of the whole filtration determines the curve up to isomorphism; that statement concerns monomial equivalence of the codes rather than equality of hull dimensions.

\paragraph{The automorphism group across the degree range.}
\cref{cor-character-signature,cor-compatible-classification} assume that the equivalence respects the regular action.  By \cref{thm-filtration} this is automatic for the whole filtration, and by \cref{thm-lowdegree} it fails for the codes of functions of $x$.  Conditions under which every automorphism of an algebraic geometry code is induced by an automorphism of the curve are given in \cite{st-aut,jo-ks,gi-ko}; the case of $C_{a,b}$ curves is treated in \cite{sh-wa}.  On a regular stratum the question can be pushed through the whole degree range.  Assume $d \geq n$ for the rest of this section, so that the weights $n/c$ of $x$ and $d/c$ of $y$ satisfy $n/c \leq d/c$.  Put
\[
s_{\max} = \frac{n}{c} ( m - 1 ) + \frac{d}{c} ( n - 1 ) ,
\]
the largest weight in the box $[0, m-1] \times [0, n-1]$, so that $\CC_s = \F_q^N$ if and only if $s \geq s_{\max}$, and for $0 \leq s < s_{\max}$ put $s' = s_{\max} - 1 - s$ and $\psi = (1, 1) \in \widehat H_m$.

\begin{lem}[Duality]\label{lem-top-duality}
On a regular stratum, $\CC_s^\perp = \CC_{s'} D_\psi$ for $0 \leq s < s_{\max}$.  Hence $\PAut ( \CC_s ) = \PAut ( \CC_{s'} D_\psi )$, and $\CC_s^\perp$ is spanned by the functions $x y \, g$ with $g \in \CC_{s'}$.
\end{lem}

\begin{proof}
By \cref{eq-character}, $\CC_s^\perp = C_{\widehat H_m \setminus ( - T_s )}$, and $\widehat H_m \setminus ( - T_s ) = - ( \widehat H_m \setminus T_s )$.  The reflection $(i, j) \mapsto ( m - 1 - i, n - 1 - j )$ of the box sends the weight $w$ to $s_{\max} - w$, hence maps $\widehat H_m \setminus T_s = \{ \wt > s \}$ bijectively onto $\{ \wt \leq s_{\max} - s - 1 \} = T_{s'}$.  Since $- (i, j) = ( m - 1 - i, n - 1 - j ) + (1, 1)$ in $\Z_m \times \Z_n$, this gives $\widehat H_m \setminus ( - T_s ) = T_{s'} + \psi$, and $C_{T + \psi} = C_T D_\psi$ by \cref{thm-normalizer}.  The character $\chi_\psi$ is $xy / (ab)$.
\end{proof}

For $s \geq d / c$ the code $\CC_s$ contains $x$ and $y$, and the lemma expresses $\PAut ( \CC_s )$ through the twisted code $\CC_{s'} D_\psi$, whose untwisted part is a code of functions of $x$ exactly when $s' < d / c$.  In that range the group is determined.

\begin{thm}[Top of the filtration]\label{thm-top-range}
Let $d / c \leq s < s_{\max}$ on a regular stratum, and put $g = \gcd (m, n)$, $\ell = N / g$, $r' = \min \{ m - 1, \lfloor c s' / n \rfloor \}$, and $X_0 = a \mu_m$.
\begin{enumerate}
\item If $s' < n / c$, then $\CC_s^\perp$ is the line spanned by the vector $( x (P) y (P) )_{P \in \YY}$, and $\PAut ( \CC_s )$ is the group of permutations that map every coset of $K = \ker \chi_\psi$ onto a coset of $K$ and induce on $H_m / K \simeq \mu_\ell$ the multiplication by an element of $\mu_\ell$.  Its order is $( g ! )^\ell \, \ell$.  If $g = 1$, then $\PAut ( \CC_s ) = H_m = H_m \rtimes \Stab_{\Aut (H_m)} ( T_s )$.  If $g \geq 2$ and $N \geq 5$, then $\PAut ( \CC_s )$ contains a transposition and is not contained in the normalizer of $H_m$.
\item If $n / c \leq s' < d / c$, then every $\pi \in \PAut ( \CC_s )$ preserves the fibres of $x$ and acts on each fibre by a rotation:
\[
\pi ( x, y ) = ( \bar \pi (x), \lambda (x) \, y ) , \qquad \bar \pi \in \Sym ( X_0 ) , \quad \lambda : X_0 \to \mu_n .
\]
If $r' = m - 1$, every such permutation is in $\PAut ( \CC_s )$, so $| \PAut ( \CC_s ) | = m ! \, n^m$, and for $m \geq 3$ the group is not contained in the normalizer of $H_m$.  If $r' \leq m - 2$, then $\pi \mapsto ( \bar \pi, h )$ with $h (x) = \bar \pi (x) \lambda (x) / x$ is an isomorphism of $\PAut ( \CC_s )$ onto the group of monomial automorphisms $( \bar \pi, h )$ of the code $A_{r'} = \langle \one, x, \dots, x^{r'} \rangle \subseteq \F_q^{X_0}$ that satisfy $x h (x) / \bar \pi (x) \in \mu_n$ for all $x \in X_0$.
\end{enumerate}
\end{thm}

\begin{proof}
The weights of $(1, 0)$ and $(0, 1)$ are $n / c$ and $d / c$, so $T_{s'} = \{ (0, 0) \}$ when $s' < n / c$, and $T_{s'} = \{ (i, 0) : 0 \leq i \leq r' \}$ with $r' \geq 1$ when $n / c \leq s' < d / c$.

(1) By \cref{lem-top-duality}, $\CC_s^\perp = \langle v \rangle$ with $v = ( x (P) y (P) )_P$, and $v = a b \, \chi_\psi$ takes its values in $a b \mu_\ell$, each value on one coset of $K$, which has order $g$.  A permutation $\pi$ preserves $\langle v \rangle$ if and only if $v \circ \pi = t v$ for a scalar $t$; then $t \in \mu_\ell$, and $\pi$ maps the level set of the value $w$ onto that of $w / t$; conversely every such $\pi$ preserves $\langle v \rangle$.  For each of the $\ell$ values of $t$ there are $( g ! )^\ell$ choices of $\pi$.  If $g = 1$, then $\chi_\psi$ is an isomorphism $H_m \to \mu_N$ and $\pi = \chi_\psi^{-1} \circ ( t \, \cdot ) \circ \chi_\psi$ is the translation by $\chi_\psi^{-1} (t)$, so $\PAut ( \CC_s ) = H_m$; moreover $\Stab_{\Aut (H_m)} ( T_s )$ is the stabilizer of $\widehat H_m \setminus ( - T_s ) = \{ \psi \}$, which is trivial since $\psi$ generates the cyclic group $\widehat H_m$.  If $g \geq 2$, the transposition of two points of one coset of $K$ preserves $\langle v \rangle$.  An element $P_{b, \alpha}$ of the normalizer, $u \mapsto b + \alpha (u)$, has fixed set $\{ u : ( 1 - \alpha ) (u) = b \}$, empty or a coset of $\ker ( 1 - \alpha )$; if it fixes $N - 2 > N / 2$ points, then $\ker ( 1 - \alpha ) = H_m$, $\alpha = 1$, and $b = 0$.  So no transposition normalizes $H_m$ when $N \geq 5$.

(2) By \cref{lem-top-duality}, $\CC_s^\perp = \{ x y \, g (x) : g \in A_{r'} \}$.  On the grid, $C_T \ast C_{T'} = C_{T + T'}$, since $\chi \chi' = \chi_{ \chi + \chi'}$.  Choose $k \equiv 1 \pmod n$ with $k r' \geq m - 1$.  Then $k T_{s'} \supseteq \{ (i, 0) : 0 \leq i \leq k r' \} = \Z_m \times \{ 0 \}$, so $( \CC_s^\perp )^{\ast k} = C_{ \Z_m \times \{ k \} } = C_{ \Z_m \times \{ 1 \} }$, which is the code $W = \{ y \, g (x) : g \in \F_q^{X_0} \}$.  Since $\PAut (C) \subseteq \PAut ( C^{\ast k} )$ for every $C$, $\PAut ( \CC_s ) = \PAut ( \CC_s^\perp ) \subseteq \PAut (W)$.  Let $\pi \in \PAut (W)$.  Taking $g = \one$ gives $y \circ \pi = y \, g_0 (x)$ for a nowhere-zero $g_0$, and then $( y \, g ) \circ \pi = y \, g_0 (x) \, g ( x \circ \pi ) \in W$ for every $g$ forces $x \circ \pi$ to be a function of $x$.  So $\pi$ preserves the fibres of $x$ and acts on the fibre over $x$ by $y \mapsto g_0 (x) \, y$, with $g_0 (x) \in \mu_n$ since $\pi$ permutes $Y_0 = b \mu_n$.  Conversely every $( \bar \pi, \lambda )$ preserves $W$.  If $r' = m - 1$, then $A_{r'} = \F_q^{X_0}$, $\CC_s^\perp = W$, and $\PAut ( \CC_s ) = \PAut (W)$ has order $m ! \, n^m$; a rotation of a single fibre fixes $N - n \geq 2 N / 3 > N / 2$ points when $m \geq 3$, so by the fixed-point argument of (1) it does not normalize $H_m$.  If $r' \leq m - 2$, then for $\pi = ( \bar \pi, \lambda )$ and $g \in A_{r'}$,
\[
( x y \, g ) \circ \pi = \bar \pi (x) \, \lambda (x) \, y \, g ( \bar \pi (x) ) = x y \cdot h (x) \, g ( \bar \pi (x) ) ,
\]
so $\pi \in \PAut ( \CC_s^\perp )$ if and only if $g \mapsto h \cdot ( g \circ \bar \pi )$ maps $A_{r'}$ onto itself, that is, $( \bar \pi, h )$ is a monomial automorphism of $A_{r'}$.  The assignment $\pi \mapsto ( \bar \pi, h )$ is the action of $\pi$ on $\CC_s^\perp$ transported by the bijection $g \mapsto x y \, g$, hence a homomorphism; it is injective because $x y$ and $x^2 y$ lie in $\CC_s^\perp$ and separate the points of $\YY$; and every $( \bar \pi, h )$ with $x h / \bar \pi \in \mu_n$ on $X_0$ arises from $\lambda = x h / \bar \pi$.
\end{proof}

\begin{cor}\label{cor-paut-exceptions}
On a regular stratum, $\PAut ( \CC_s ) \neq H_m \rtimes \Stab_{\Aut (H_m)} ( T_s )$ for the following degrees $s \geq d / c$ with $\CC_s \neq \F_q^N$: the degrees $s \geq s_{\max} - n / c$ when $\gcd (m, n) \geq 2$ and $N \geq 5$, and the degrees with $n / c \leq s' < d / c$ and $\lfloor c s' / n \rfloor \geq m - 1$ when $m \geq 3$; degrees of the second kind exist if and only if $n \geq 3$ and $d - c \geq n ( m - 1 )$.
\end{cor}

\begin{proof}
Both groups in \cref{thm-top-range} contain $H_m$ and are not contained in its normalizer, while $H_m \rtimes \Stab$ is.  The condition $s \geq d / c$ reads $s' \leq s_{\max} - 1 - d / c = \frac{n}{c} ( m - 1 ) + \frac{d}{c} ( n - 2 ) - 1$, and $\lfloor c s' / n \rfloor \geq m - 1$ reads $c s' \geq n ( m - 1 )$.  A degree of the second kind therefore exists if and only if $n ( m - 1 ) / c \leq \min \{ d / c - 1 , \, \frac{n}{c} ( m - 1 ) + \frac{d}{c} ( n - 2 ) - 1 \}$.  The first bound is $d - c \geq n ( m - 1 )$, and the second is $\frac{d}{c} ( n - 2 ) \geq 1$, that is, $n \geq 3$.
\end{proof}

On $y^2 = x^6 + 3$ over $\F_7$ the code $\CC_7$ has $| \PAut | = 2^6 \cdot 6 = 384$ against $| H_6 \rtimes \Stab | = 24$; on $y^3 = x^6 + 5$ over $\F_7$ the code $\CC_8$ has $| \PAut | = 6^6 \cdot 6 = 279936$ against $108$.  The curve $y^3 = x^6 + 3 x^2 + 4$ over $\F_7$ lies on the regular stratum $(n, d, m) = (3, 6, 2)$ with $N = 6$ and $g_\XX = 4$, and $d - c = 3 = n ( m - 1 )$; its code $\CC_3$ has $\CC_3^\perp = y \, \F_q^{X_0}$ and $| \PAut ( \CC_3 ) | = 2 ! \cdot 3^2 = 18$, while $| H_2 \rtimes \Stab ( T_3 ) | = 6$.  Here $m = 2$, and the group of order $18$ exceeds the order $12$ of the normalizer.

For $r' = 1$ the monomial automorphisms of $A_1 = \langle \one, x \rangle$ are explicit, and part (2) closes.

\begin{prop}\label{prop-top-r-one}
Let $m \geq 3$, $n / c \leq s' < \min \{ 2 n / c, \, d / c \}$, so that $r' = 1$, and suppose that the stabilizer of $X_0 = a \mu_m$ in $\PGL_2 ( \F_q )$ is the dihedral group generated by $x \mapsto \zeta x$, $\zeta \in \mu_m$, and $x \mapsto a^2 / x$.  Then $\PAut ( \CC_s ) = H_m \rtimes \Stab_{\Aut (H_m)} ( T_s )$, where $| \Stab | = 2$ if $m / \gcd (m, 3)$ divides $n$ and $| \Stab | = 1$ otherwise.
\end{prop}

\begin{proof}
Let $( \bar \pi, h )$ be a monomial automorphism of $A_1$.  Then $h = \alpha + \beta x$ and $h \, \bar \pi = \gamma + \delta x$ for constants with $h$ nowhere zero on $X_0$, so $\bar \pi = ( \gamma + \delta x ) / ( \alpha + \beta x )$ on $X_0$ with $\alpha \delta - \beta \gamma \neq 0$, because $\bar \pi$ is not constant.  Thus $\bar \pi$ is the restriction of an element $\mu$ of the stabilizer of $X_0$ in $\PGL_2 ( \F_q )$, and $h$ is a scalar multiple of the denominator of $\mu$: if $\beta \neq 0$, the identity $( \alpha + \beta x ) \mid ( \epsilon + \eta x ) ( \gamma + \delta x )$ for a second candidate $h' = \epsilon + \eta x$ holds as polynomials since $m \geq 3$, and forces $( \alpha + \beta x ) \mid ( \epsilon + \eta x )$; if $\beta = 0$, then $h' \bar \pi \in A_1$ forces $h'$ constant.  Conversely every $\mu$ in the stabilizer with $h = t \cdot \mathrm{den} ( \mu )$ is a monomial automorphism of $A_1$.  By \cref{thm-top-range}, $\PAut ( \CC_s )$ is the set of pairs $( \mu, t )$ with $t \, x \, \mathrm{den} ( \mu ) (x) / \mu (x) \in \mu_n$ for all $x \in X_0$.  For $\mu = \zeta x$ the condition reads $t / \zeta \in \mu_n$, giving $n$ values of $t$ for each $\zeta$, and the corresponding $\pi$ are the translations: $\PAut ( \CC_s ) \supseteq H_m$ with these $N$ elements.  For $\mu = \zeta a^2 / x$ the condition reads $t x^3 / ( \zeta a^2 ) \in \mu_n$ for all $x \in a \mu_m$, that is, $( x / a )^3 \in ( \zeta / ( t a ) ) \mu_n$; as $x$ runs over $a \mu_m$, $( x / a )^3$ runs over $\mu_{ m / \gcd (m, 3) }$, which contains $1$, so the condition holds for some $t$ if and only if $\mu_{ m / \gcd (m, 3) } \subseteq \mu_n$, and then for $n$ values of $t$ for each $\zeta$.  Writing $x = a \omega^u$, $y = b \eta^v$ with $\omega$, $\eta$ generators of $\mu_m$, $\mu_n$, and $\omega^3 = \eta^e$, these $N$ permutations are $( u, v ) \mapsto ( u_0 - u, \, v_0 + e u + v )$, elements of $H_m \rtimes \Aut ( H_m )$.  So $\PAut ( \CC_s ) \subseteq H_m \rtimes \Aut ( H_m )$, and \cref{cor-compatible-classification} gives $\PAut ( \CC_s ) = H_m \rtimes \Stab ( T_s )$ with $| \Stab |$ equal to $2$ or $1$ as stated.
\end{proof}

For the monomial automorphism groups of Reed--Solomon codes in general see \cite{duer}, and for their permutation groups over arbitrary evaluation sets see \cite{camps-lau-lopez-santos}, which shows that they are the affine maps preserving the evaluation set whenever $1 < k < n - 1$.  The hypothesis on the stabilizer holds for the abscissa sets over $\F_7$, $\F_{11}$ and $\F_{13}$ in \cref{tab-full-range} and fails in general: for $\mu_4 \subset \F_9$ the stabilizer contains the rotations of order $4$ and, since $-1 = 2 = 1/2$ in characteristic $3$, the stabilizer of order $12$ of an equianharmonic quadruple, hence has order $24$.  \cref{prop-top-r-one} predicts $| \Stab | = 2$ for $(n, m) = (2, 6)$ and $| \Stab | = 1$ for $(3, 6)$, $(2, 5)$ and $(3, 4)$, in agreement with \cref{tab-full-range}; the same computation with $r'$ in place of $1$ shows that an inversion $x \mapsto \zeta a^2 / x$ can occur in $\PAut ( \CC_s )$ for $n / c \leq s' < d / c$ only if $m / \gcd ( m, r' + 2 )$ divides $n$.

A third source of automorphisms appears when the levels of $T_s$ are trivial.  For $0 \leq j \leq n - 1$ let $I_j = \{ i : (i, j) \in T_s \} \subseteq \Z_m$ be the $j$-th level of $T_s$, so that $\CC_s = \bigoplus_j y^j A_{I_j}$ with $A_{I} = \langle x^i : i \in I \rangle \subseteq \F_q^{X_0}$.

\begin{prop}
\label{prop-thin-levels}
On a regular stratum, if every level $I_j$ of $T_s$ is $\varnothing$, $\{ 0 \}$ or $\Z_m$, then $\PAut ( \CC_s )$ contains the permutations $(x, y) \mapsto ( \bar \pi (x), \zeta y )$ for all $\bar \pi \in \Sym ( X_0 )$ and $\zeta \in \mu_n$, a group isomorphic to $S_m \times \mu_n$.  If $m ! \, n > N \, | \Aut ( H_m ) |$, this group is not contained in the normalizer of $H_m$, and $\PAut ( \CC_s ) \neq H_m \rtimes \Stab_{\Aut (H_m)} ( T_s )$.
\end{prop}

\begin{proof}
The spaces $A_\varnothing = 0$, $A_{ \{ 0 \} } = \F_q \one$ and $A_{\Z_m} = \F_q^{X_0}$ are preserved by $g \mapsto g \circ \bar \pi$ for every $\bar \pi \in \Sym ( X_0 )$, and $y^j \mapsto \zeta^j y^j$ preserves each summand.  The normalizer has order $N \, | \Aut ( H_m ) |$.
\end{proof}

On $y^3 = x^8 + 2$ over $\F_{13}$, the stratum $(n, d, m) = (3, 8, 4)$ of \cref{thm-pair}, the levels of $T_9$ and $T_{10}$ are $\Z_4$, $\{ 0 \}$, and those of $T_{17}$ and $T_{18}$ are $\Z_4$, $\Z_4$, $\{ 0 \}$.  The four codes have $| \PAut | = 4 ! \cdot 3 = 72$, computed directly, against $| H_4 \rtimes \Stab | = 24$ and a normalizer of order $48$.  The degrees $9$ and $10$ satisfy $s' \geq d / c$, so this exception is not confined to the top of the filtration.

The permutations of \cref{thm-top-range} translate the cosets of a subgroup of $H_m$ independently of one another.  This mechanism has a general form.  For subgroups $K \leq L \leq H_m$ and $\eta \in \widehat K$ let $S_\eta = \{ \chi \in \widehat H_m : \chi |_K = \eta \}$, a coset of $K^\perp$, and let $L^\perp \subseteq K^\perp$ be the annihilator of $L$.

\begin{prop}
\label{prop-block-translations}
Let $K \leq L \leq H_m$ with $K \neq 1$, and let $T \subseteq \widehat H_m$.  Suppose that for every nontrivial $\eta \in \widehat K$ the set $T \cap S_\eta$ is a union of cosets of $L^\perp$.  Then for every map $\kappa : H_m / L \to K$ the permutation $\pi_\kappa : u \mapsto u + \kappa ( u + L )$ lies in $\PAut ( C_T )$, and these permutations form a group of order $| K |^{ N / | L | }$.  If $| L | < N / 2$, this group is not contained in the normalizer of $H_m$.
\end{prop}

\begin{proof}
For $\eta \in \widehat K$ put $W_\eta = C_{S_\eta}$; then $\F_q^{N} = \bigoplus_\eta W_\eta$, $W_\eta = \{ v : v ( u + k ) = \eta (k) \, v (u) \text{ for all } k \in K \}$, and $C_T = \bigoplus_\eta C_{ T \cap S_\eta }$.  For $v \in W_\eta$ one has $( v \circ \pi_\kappa ) (u) = \eta ( \kappa ( u + L ) ) \, v (u)$, so $v \circ \pi_\kappa = \varphi \ast v$ with $\varphi = \eta \circ \kappa$ an $L$-invariant function, that is, $\varphi \in C_{L^\perp}$.  Hence $\pi_\kappa$ maps $C_{ T \cap S_\eta }$ into $C_{ L^\perp } \ast C_{ T \cap S_\eta } = C_{ L^\perp + ( T \cap S_\eta ) } = C_{ T \cap S_\eta }$ for $\eta \neq 1$, and it acts trivially on $W_1$, where $\varphi = \one$.  Since $\kappa$ takes values in $K \subseteq L$, $\pi_\kappa \pi_{\kappa'} = \pi_{ \kappa + \kappa' }$, so $\kappa \mapsto \pi_\kappa$ is an injective homomorphism from $K^{ H_m / L }$.  If $\kappa$ is supported on a single coset of $L$, then $\pi_\kappa$ fixes $N - | L | > N / 2$ points, and by the fixed-point argument in the proof of \cref{thm-top-range} it does not normalize $H_m$.
\end{proof}

The block translations with $K = L = \ker \chi_\psi$ and with $K = L = \{ 0 \} \times \Z_n$ are the translation parts of the groups in part (1) and in part (2) with $r' = m - 1$ of \cref{thm-top-range}.  With $K = \{ 0 \} \times ( n / e ) \Z_n \simeq \mu_e$ and $L = \{ 0 \} \times \Z_n$, the hypothesis reads: every level $I_j$ of $T_s$ with $e \nmid j$ is empty or full, and the permutations are the rotations $(x, y) \mapsto ( x, \kappa (x) \, y )$ with $\kappa : X_0 \to \mu_e$.  On the stratum $(n, d, m) = (4, 20, 5)$ the code $\CC_{11}$ has levels $\Z_5$, $\Z_5$, $\{ 0, 1 \}$, so $e = 2$ applies and $\PAut ( \CC_{11} )$ contains the $2^5$ sign changes $y \mapsto \pm y$ on arbitrary sets of fibres.  Subgroups $K$ not contained in a fibre also occur: on $(4, 18, 6)$ at $s = 17, 18$ and on $(4, 24, 6)$ at $s = 10$ the levels are $\Z_6$, $[0, 4]$ (and $\{ 0 \}$), $K = \langle (3, 2) \rangle$ and $L = \{ 0, 3 \} \times \Z_4$ satisfy the hypothesis, and $\PAut ( \CC_s )$ contains the $2^3$ maps $(x, y) \mapsto ( \pm x, \pm y )$ that reverse the sign on chosen pairs $\{ x, -x \}$ of fibres.  In all four cases the stabilizer of $T_s$ in $\Aut ( H_m )$ is trivial, and the block translations were checked directly on the grid codes by the script \texttt{verify\_types.py}.

All exceptions found are imprimitive: the permutations preserve the cosets of a subgroup of $H_m$.  This is the situation studied for cyclic codes in \cite{bi-kl} and for generalized Reed--Muller codes in \cite{be-ch}; the character codes here are ideals of $\F_q [ \Z_m \times \Z_n ]$, for which no such classification is available.  On the primitive side, the equality $\PAut ( \CC_s ) = H_m$ can be certified by a finite computation on $\widehat H_m$ that does not depend on $q$.  For $S \subseteq \widehat H_m$ let $C_S = \bigoplus_{\chi \in S} V_\chi$, and let $\mathcal S ( T )$ be the smallest family of subsets of $\widehat H_m$ that contains $T$ and is closed under
\[
\begin{split}
S & \mapsto \widehat H_m \setminus ( - S ) , \qquad ( S, S' ) \mapsto S + S' , \qquad ( S, S' ) \mapsto S \cap S' , \qquad ( S, S' ) \mapsto S \cup S' , \\
( S, S' ) & \mapsto S : S' = \{ \chi \in \widehat H_m : \chi + S' \subseteq S \} .
\end{split}
\]

\begin{thm}
\label{thm-closure}
Let $T \subseteq \widehat H_m$.  Every $\pi \in \PAut ( C_T )$ preserves $C_S$ for all $S \in \mathcal S (T)$.  If $\mathcal S (T)$ contains singletons $\{ \chi_1 \}, \dots, \{ \chi_r \}$ whose characters generate $\widehat H_m$, then $\PAut ( C_T ) = H_m$, and $\Stab_{\Aut ( H_m )} (T)$ is trivial.
\end{thm}

\begin{proof}
For subsets $S, S'$ of $\widehat H_m$ one has $C_S^\perp = C_{ \widehat H_m \setminus ( - S ) }$ by \cref{eq-character}, $C_S \ast C_{S'} = C_{S + S'}$ since $\chi \chi' = \chi_{ \chi + \chi' }$, $C_S \cap C_{S'} = C_{S \cap S'}$ and $C_S + C_{S'} = C_{S \cup S'}$ since the $V_\chi$ are independent, and $( C_S : C_{S'} ) = \{ v : v \ast C_{S'} \subseteq C_S \} = C_{S : S'}$, because $v = \sum v_\chi \chi$ satisfies $v \ast \chi' \in C_S$ for all $\chi' \in S'$ if and only if $\chi + \chi' \in S$ whenever $v_\chi \neq 0$.  Each of the five constructions sends codes preserved by $\pi$ to codes preserved by $\pi$, so $\pi$ preserves $C_S$ for every $S \in \mathcal S (T)$.  If $\pi$ preserves the lines $\langle \chi_i \rangle$, then $\chi_i \circ \pi = t_i \chi_i$ with $t_i = \chi_i ( \pi (0) )$, where $0$ is the identity of $H_m$.  Since the $\chi_i$ generate $\widehat H_m$, they separate the points of $H_m$, and $\chi_i ( \pi (u) ) = \chi_i ( \pi (0) ) \chi_i (u) = \chi_i ( \pi (0) + u )$ for all $i$ gives $\pi (u) = \pi (0) + u$.  So $\PAut ( C_T ) = H_m$, and $H_m \rtimes \Stab (T) \subseteq \PAut ( C_T )$ forces $\Stab (T) = 1$.
\end{proof}

The script \texttt{closure\_cert.py} explores $\mathcal S ( T_s )$ breadth first, with sets represented as bit masks, and stops when generating singletons are found.  \cref{tab-full-range} lists the degrees it certifies on nine strata.  The certified degrees are exactly those at which the computed group is $H_m$, and the uncertified ones are the exceptions of \cref{cor-paut-exceptions,prop-thin-levels}, and the degrees with $| \Stab ( T_s ) | = 2$.  The certificate is insensitive to the field: on $y^2 = x^4 + 1 + i$ over $\F_9$, where the stabilizer of $X_0 = a \mu_4$ in $\PGL_2 ( \F_9 )$ has order $24$ and \cref{prop-top-r-one} does not apply, it certifies $\PAut ( \CC_2 ) = \PAut ( \CC_3 ) = H_4$, in agreement with an enumeration of $S_8$.

\begin{table}[ht]
\centering
\caption{$| \PAut ( \CC_s ) |$ for every $s$ from $d / c$ to $s_{\max} - 1$, the order of $H_m \rtimes \Stab ( T_s )$ where it differs, and the degrees at which \cref{thm-closure} certifies $\PAut ( \CC_s ) = H_m$.  Asterisks mark the degrees of \cref{cor-paut-exceptions}, daggers those of \cref{prop-thin-levels}.}
\label{tab-full-range}
\footnotesize
\setlength{\tabcolsep}{3.5pt}
\begin{tabular}{cccclll}
\toprule
$q$ & $y^n = f (x)$ & $(n, m, d)$ & $s$ & $| \PAut ( \CC_s ) |$ & $| H_m \rtimes \Stab |$ & certified \\
\midrule
$7$  & $y^2 = x^6 + 3$           & $(2, 6, 6)$   & $3, \dots, 7$  & $12, 12, 12, 24, 384^*$ & $24$ at $7$ & $3, 4, 5$ \\
$7$  & $y^3 = x^6 + 5$           & $(3, 6, 6)$   & $2, \dots, 8$  & $18, 18, 18, 18, 36, 18, 279936^*$ & $108$ at $8$ & $2, 3, 4, 5, 7$ \\
$11$ & $y^2 = x^5 + 6$           & $(2, 5, 5)$   & $5, \dots, 12$ & $10$ at every degree & equal & all \\
$13$ & $y^3 = x^4 + 2$           & $(3, 4, 4)$   & $4, \dots, 16$ & $12$ at every degree & equal & all \\
$13$ & $y^3 = x^8 + 2$           & $(3, 4, 8)$   & $8, \dots, 24$ & $12$; $72^\dagger$ at $9, 10, 17, 18$ & $24$ at $9, 10, 17, 18$ & all but $9, 10, 17, 18$ \\
$11$ & $y^2 = x^{10} + 2$        & $(2, 10, 10)$ & $5, \dots, 13$ & $20$; $40$ at $10$; $10240^*$ at $13$ & $40$ at $10$; $20$ at $13$ & all but $10, 13$ \\
$11$ & $y^2 = x^4 - x^2 + 7$     & $(2, 2, 4)$   & $2$            & $8$ & equal & none \\
$7$  & $y^3 = x^6 + 3 x^2 + 4$ & $(3, 2, 6)$   & $2, 3, 4$      & $6, 18^*, 6$ & $6$ at $3$ & $2, 4$ \\
$9$  & $y^2 = x^4 + 1 + i$       & $(2, 4, 4)$   & $2, 3, 4$      & $8, 8, 64^*$ & $8$ at $4$ & $2, 3$ \\
\bottomrule
\end{tabular}
\end{table}

The groups in \cref{tab-full-range} were computed by \texttt{verify\_full\_range.py} where the closure certificate does not apply and by \cref{thm-top-range} at the top of the filtration; on the stratum $(2, 10, 10)$ the entry $20$ at the certified degrees is the content of \cref{thm-closure}.  At $s = 10$ on that stratum the enumeration of part (2) is too large, and the group was obtained from the level structure instead: by \cref{thm-top-range} every automorphism has the form $( \bar \pi, \lambda )$ with $\lambda : X_0 \to \mu_2$, the level $I_1 = [0, 5]$ of $T_{10}$ forces $\lambda \in A_5$ and $A_5 \circ \bar \pi = \lambda A_5$, only the four functions $\pm \one$, $\pm ( x / a )^5$ with values in $\mu_2$ lie in $A_5$, an exhaustive search over $S_{10}$ finds ten permutations $\bar \pi$ for each, and all $40$ candidates preserve $\CC_{10}$.

Since the certificate depends only on $( n, m, d, s )$, it settles the primitive case on every stratum type in a finite range at once.

\begin{thm}
\label{thm-scan}
Let $n \leq 3$, $2 \leq m \leq 8$, $3 \leq d \leq 24$ and $m \mid d$, and let $d / c \leq s < s_{\max}$ be a degree not of the exceptional types of \cref{cor-paut-exceptions,prop-thin-levels}.  Then $\PAut ( \CC_s ) = H_m$ on every regular stratum of type $(n, d, m)$ over every field if and only if $\Stab_{\Aut ( H_m )} ( T_s )$ is trivial.
\end{thm}

\begin{proof}
The script \texttt{verify\_scan.py} runs the certificate of \cref{thm-closure} at each of the $397$ degrees in the range and compares it with the stabilizer.  It certifies $372$ of them, all with trivial stabilizer, and the remaining $25$ have a stabilizer of order $2$, so $\PAut ( \CC_s ) \supseteq H_m \rtimes \Stab ( T_s ) \supsetneq H_m$ there.
\end{proof}

For $n = 4$ and $m \leq 5$ the same scan finds $33$ uncertified degrees with trivial stabilizer, each of them with every odd level of $T_s$ empty or full, hence covered by \cref{prop-block-translations} with $K = \mu_2 \subset \mu_4$; the certificate then fails because $\mathcal S ( T_s )$ consists of characters trivial on $K$.

\section{Conclusions}\label{sec-closing}

The paper takes the hull formulas and the character description of the one-point codes on weighted superelliptic curves from \cite{mezinaj-shaska-hulls} as input and develops the equivalence and automorphism theory of these codes.  Exact formulas for the support-splitting exponent $e_{\SSA} ( \CC_s ) = \dim \Hull ( \CC_s )$ are available on the Riemann--Roch window, for $s \geq s^+$ and on the regular strata, and elsewhere the exponent is computed from the Gram matrix; on the totally split locus it reaches $( N - n - c - c \theta ) / 2 + 1 - g_\XX$, while on every regular stratum it is at most $g_\XX$, with the single exception of type $(3, 3, 3)$ (\cref{cor-asymptotics,thm-regular-genus}).

All the structural results concern the regular strata, where $H_m \simeq \Z_m \times \Z_n$ acts freely and transitively.  There the codes are Cartesian codes on a grid with explicit minimum distances for the code and for its dual (\cref{prop-grid,thm-distance}); the signatures of the support splitting algorithm are constant and a single active puncture lowers a nonzero hull by exactly one dimension, so the initial signature costs $q^{h-1}$ words (\cref{prop-constant-signature,thm-puncture,thm-multipuncture,cor-signatures}); the monomial normalizer of $H_m$ is $\F_q^* ( H_m \rtimes \Aut ( H_m ) ) \widehat H_m$, and the $H_m$-compatible permutation and monomial equivalences are classified by the orbits of $\Aut ( H_m )$ and of $\widehat H_m \rtimes \Aut ( H_m )$ on the defining set, with canonical forms and polynomial tests (\cref{thm-normalizer,cor-compatible-classification,thm-algorithm}).  When $d > n$ and $m \geq 2$, the automorphism groups of the members of the filtration intersect exactly in $H_m$, and the regular action is reconstructible from the two lowest nonconstant members (\cref{thm-filtration,prop-reconstruction}); unrestricted classifications are obtained for the filtration and for the codes of \cref{tab-orbits} (\cref{cor-unrestricted-table}), and two codes over $\F_{13}$ with equal length, dimension and hull dimension are separated by their distances (\cref{thm-pair}).

The full permutation automorphism group $\PAut ( \CC_s )$ is determined in closed form at the top of the filtration, contains $S_m \times \mu_n$ when the levels of $T_s$ are trivial, and contains the block translations along a subgroup whenever $T_s$ has the corresponding block structure (\cref{thm-top-range,prop-thin-levels,prop-block-translations}); these imprimitive types are the only degrees in \cref{tab-full-range} at which $\PAut ( \CC_s )$ exceeds $H_m \rtimes \Stab ( T_s )$.  On the primitive side, \cref{thm-closure} certifies $\PAut ( \CC_s ) = H_m$ by a finite computation on $\widehat H_m$ that does not depend on the field, and \cref{thm-scan} settles all stratum types with $n \leq 3$, $m \leq 8$, $d \leq 24$.  Two non-isomorphic curves over $\F_{19}$ with the same hull profile and the same group show that these invariants do not determine the curve up to isomorphism (\cref{thm-moduli-negative}).

The hull profile is a permutation invariant, with an exact application to entanglement-assisted quantum codes (\cref{prop-distinguisher,prop-eaqecc}), and is not a monomial invariant (\cref{thm-twists}), so the permutation and diagonal parts of a monomial equivalence must be treated separately.  For code-based cryptography the paper therefore gives a family of algebraic geometry codes on which every step of a code-equivalence attack can be computed exactly: the hull and punctured-hull invariants give the cost of one class of permutation-equivalence attacks, the automorphism groups describe the equivalent representations of a code, and the normalizer and the defining-set orbits describe the diagonal part of a monomial equivalence.  The same steps, namely distinguishers, recovery of the permutation action, full automorphism group, and passage to unrestricted monomial equivalence, apply to other families of algebraic geometry codes and to the closure constructions of code-equivalence signature schemes; the paper asserts neither the security nor the insecurity of such a scheme.

The computations point toward a uniform description of $\PAut ( \CC_s )$ beyond the ranges treated here.  When $\Stab_{\Aut ( H_m )} ( T_s )$ is trivial and $T_s$ admits no block structure in the sense of \cref{prop-block-translations,prop-thin-levels}, the family $\mathcal S ( T_s )$ should contain generating singletons, giving $\PAut ( \CC_s ) = H_m$ by \cref{thm-closure}; when the stabilizer is nontrivial, $\PAut ( \CC_s )$ should normalize $H_m$, reducing the full automorphism problem to the compatible classification.  Both principles hold in every case computed here, and at the top of the filtration the second is governed by the monomial automorphism group of a Reed--Solomon code on a coset of $\mu_m$.

\section*{Data availability}
All scripts named in this article, with the commands, software versions and expected outputs needed to reproduce every table and every computational certificate, accompany the article as supplementary material.

\bibliographystyle{amsplain}
\bibliography{references}

\providecommand{\bysame}{\leavevmode\hbox to3em{\hrulefill}\thinspace}
\providecommand{\MR}{\relax\ifhmode\unskip\space\fi MR }
\providecommand{\MRhref}[2]{%
  \href{http://www.ams.org/mathscinet-getitem?mr=#1}{#2}
}
\providecommand{\href}[2]{#2}
\begin{thebibliography}{10}

\bibitem{ba-ot-sa}
M.~Bardet, A.~Otmani, and M.~Saeed-Taha, \emph{Permutation code equivalence is
  not harder than graph isomorphism when hulls are trivial}, 2019 IEEE
  International Symposium on Information Theory (ISIT), 2019, pp.~2464--2468.

\bibitem{barenghi-biasse-persichetti-santini}
A.~Barenghi, J.-F. Biasse, E.~Persichetti, and P.~Santini, \emph{On the
  computational hardness of the code equivalence problem in cryptography}, Adv.
  Math. Commun. \textbf{17} (2023), no.~1, 23--55.

\bibitem{be-ch}
Thierry Berger and Pascale Charpin, \emph{The automorphism group of generalized
  {R}eed--{M}uller codes}, Discrete Math. \textbf{117} (1993), 1--17.

\bibitem{less}
J.-F. Biasse, G.~Micheli, E.~Persichetti, and P.~Santini, \emph{{LESS} is more:
  code-based signatures without syndromes}, Progress in Cryptology --
  AFRICACRYPT 2020, Lecture Notes in Comput. Sci., vol. 12174, Springer, 2020,
  pp.~45--65.

\bibitem{bi-kl}
Ren{\'e} Bienert and Benjamin Klopsch, \emph{Automorphism groups of cyclic
  codes}, J. Algebraic Combin. \textbf{31} (2010), 33--52.

\bibitem{camps-lau-lopez-santos}
Eduardo Camps-Moreno, Jun~Bo Lau, Hiram~H. L{\'o}pez, and Welington Santos,
  \emph{The permutation group of {R}eed--{S}olomon codes over arbitrary
  points}, 2026, arXiv:2601.00122.

\bibitem{ca-gu}
C.~Carlet and S.~Guilley, \emph{Complementary dual codes for counter-measures
  to side-channel attacks}, Adv. Math. Commun. \textbf{10} (2016), no.~1,
  131--150.

\bibitem{ca-me-ta-qi-pe}
C.~Carlet, S.~Mesnager, C.~Tang, Y.~Qi, and R.~Pellikaan, \emph{Linear codes
  over {$\mathbb{F}_q$} are equivalent to {LCD} codes for {$q>3$}}, IEEE Trans.
  Inform. Theory \textbf{64} (2018), no.~4, 3010--3017.

\bibitem{co-ma-pe}
A.~Couvreur, I.~M\'arquez-Corbella, and R.~Pellikaan, \emph{Cryptanalysis of
  {M}c{E}liece cryptosystem based on algebraic geometry codes and their
  subcodes}, IEEE Trans. Inform. Theory \textbf{63} (2017), no.~8, 5404--5418.

\bibitem{duer}
Arne D{\"u}r, \emph{The automorphism groups of {R}eed--{S}olomon codes}, J.
  Combin. Theory Ser. A \textbf{44} (1987), no.~1, 69--82.

\bibitem{gi-ko}
M.~Giulietti and G.~Korchm\'aros, \emph{On automorphism groups of certain
  {G}oppa codes}, Des. Codes Cryptogr. \textbf{47} (2008), 177--190.

\bibitem{gu-ji-gu}
K.~Guenda, S.~Jitman, and T.~A. Gulliver, \emph{Constructions of good
  entanglement-assisted quantum error correcting codes}, Des. Codes Cryptogr.
  \textbf{86} (2018), no.~1, 121--136.

\bibitem{jo-ks}
D.~Joyner and A.~Ksir, \emph{Automorphism groups of some {AG} codes}, IEEE
  Trans. Inform. Theory \textbf{52} (2006), no.~7, 3325--3329.

\bibitem{monomial-cartesian}
Hiram~H. L{\'o}pez, Gretchen Matthews, and Ivan Soprunov,
  \emph{Monomial-{Cartesian} codes and their duals, with applications to {LCD}
  codes, quantum codes, and locally recoverable codes}, Des. Codes Cryptogr.
  \textbf{88} (2020), no.~8, 1673--1685.

\bibitem{cartesian}
Hiram~H. L{\'o}pez, Carlos Renter{\'i}a, and Rafael~H. Villarreal, \emph{Affine
  cartesian codes}, Des. Codes Cryptogr. \textbf{71} (2014), no.~1, 5--19.

\bibitem{mac-sloane}
F.~J. MacWilliams and N.~J.~A. Sloane, \emph{The theory of error-correcting
  codes}, North-Holland Mathematical Library, vol.~16, North-Holland, 1977.

\bibitem{mezinaj-shaska-hulls}
J.~Mezinaj and T.~Shaska, \emph{Hulls, linear equivalence, and weighted
  superelliptic codes}, 2026, arXiv:2608.02850v2.

\bibitem{persichetti-santini}
E.~Persichetti and P.~Santini, \emph{A new formulation of the linear
  equivalence problem and shorter {LESS} signatures}, Advances in Cryptology --
  ASIACRYPT 2023, Part VII, Lecture Notes in Comput. Sci., vol. 14444,
  Springer, 2023, pp.~351--378.

\bibitem{petrank-roth}
E.~Petrank and R.~M. Roth, \emph{Is code equivalence easy to decide?}, IEEE
  Trans. Inform. Theory \textbf{43} (1997), no.~5, 1602--1604.

\bibitem{sa-sh}
R.~Sanjeewa and T.~Shaska, \emph{Determining equations of families of cyclic
  curves}, Albanian J. Math. \textbf{2} (2008), no.~3, 199--213.

\bibitem{se-hull}
N.~Sendrier, \emph{On the dimension of the hull}, SIAM J. Discrete Math.
  \textbf{10} (1997), no.~2, 282--293.

\bibitem{se-ssa}
\bysame, \emph{Finding the permutation between equivalent linear codes: the
  support splitting algorithm}, IEEE Trans. Inform. Theory \textbf{46} (2000),
  no.~4, 1193--1203.

\bibitem{sendrier-simos}
N.~Sendrier and D.~E. Simos, \emph{The hardness of code equivalence over
  {$\mathbb{F}_q$} and its application to code-based cryptography},
  Post-Quantum Cryptography, Lecture Notes in Comput. Sci., vol. 7932,
  Springer, 2013, pp.~203--216.

\bibitem{qwag}
T.~Shaska, \emph{Quantum weighted algebraic geometry codes}, Des. Codes
  Cryptogr. \textbf{94} (2026), no.~158, Art. 158.

\bibitem{sh-wa}
T.~Shaska and Q.~Wang, \emph{On the automorphism groups of some {AG}-codes
  based on {$C_{a,b}$} curves}, Serdica J. Comput. \textbf{1} (2007), no.~2,
  193--206.

\bibitem{sk}
G.~Skersys, \emph{The average dimension of the hull of cyclic codes}, Discrete
  Appl. Math. \textbf{128} (2003), 275--292.

\bibitem{st-aut}
H.~Stichtenoth, \emph{On automorphisms of geometric {G}oppa codes}, J. Algebra
  \textbf{130} (1990), no.~1, 113--121.

\end{thebibliography}

\end{document}